\documentclass[reqno]{amsart}
\usepackage{hyperref}
\usepackage{amsfonts}
\usepackage{amsthm,amsmath}
\usepackage[integrals]{wasysym}
\usepackage{enumerate}
\usepackage{fancyhdr}
\usepackage{amssymb}
\usepackage{chemarrow}
\usepackage{tikz}
\usepackage{mathtools}
\usepackage{mathrsfs}

\numberwithin{equation}{section}
\newtheorem{theorem}{Theorem}[section]
\newtheorem{lemma}[theorem]{Lemma}
\newtheorem{definition}[theorem]{Definition}
\newtheorem{proposition}[theorem]{Proposition}
\newtheorem{corollary}[theorem]{Corollary}
\theoremstyle{remark}
\newtheorem{remark}{Remark}
\allowdisplaybreaks

\def\cS{{\mathcal S}}

\def\div{ \hbox{\rm div}\,  }

\def\N{{\mathbb N}}
\def\R{{\mathbb R}}
\def\T{{\mathbb T}}
\def\Z{{\mathbb Z}}

\def\eps{\varepsilon}

\begin{document}
 \title[\hfilneg \hfil ]
{ Weak solutions of Moffatt's magnetic relaxation equations
}

 \author[Tan]{ \text{Jin Tan}}
 \address[Jin Tan]{\newline  Department of Mathematics,The Chinese University of Hong Kong, Shatin, Hong Kong, P.R. China}
 \email{jin.tan@cuhk.edu.hk}

\begin{abstract}
 We investigate the  existence of finite-energy weak solutions for a family of equations first introduced by Moffatt to model magnetic relaxation. These equations are active vector equations,   analogous to  the 3D classical Euler equations in vorticity form, and supplemented with a non-local bilinear constitutive law that depends on a regularization parameter $\gamma\geq 0.$ However,  that regularization term  does not yield any  obvious robust compactness properties  for the magnetic field,  regardless of  the value of $\gamma.$    Therefore, the global existence of  weak solutions for merely finite-energy initial data is not known and has been proposed as an open problem in  Beekie, Friedlander and Vicol \cite[Page 1337, Q4]{BFV}.
This paper partially answers the question by showing the existence and \emph{uniqueness} of global  weak solutions  for any solenoidal initial magnetic field  $B_0\in L^2(\T^d)$ in the case where  $\gamma>{d}/{2}+1.$
 Moreover, the existence and uniqueness result holds  for the borderline case $\gamma={d}/{2}+1$    if $B_0\in L^{p}(\T^d),$  for some $p>2.$    To prove these results,  a key compactness lemma  for an active vector equation is established,  which may be of independent interest.


%

\end{abstract}

\maketitle
  
 \section{Introduction}

In  Arnold's seminal work \cite{Ar66}, he developed a new set of geometric ideas for the classical incompressible Euler equations:
\begin{equation}\label{eq-Euler}
    \partial_t u+u\cdot\nabla u+\nabla P=0,\quad \div u=0.
\end{equation}
This geometric perspective views  \eqref{eq-Euler} as the geodesic equations of a right-invariant metric on the infinite-dimensional group of volume preserving diffeomorphisms. 
Since then there have been  numerous works in the literature devoted to the subject of topological hydrodynamics,  for example Ebin and Marsden \cite{EM70}, Holm et al. \cite{HMRW85},  Arnold and  Khesin \cite{AK98}.  On the other hand,  Arnold \cite{Ar74} suggested a process which demonstrates the existence of an Euler equilibrium (i.e. steady state of system \eqref{eq-Euler}) that has the same topological structure as an arbitrary divergence free magnetic field. The idea is to use the dynamics of the magnetic field to reach an Euler  equilibrium     which is topology-preserving (see Brenier \cite{Br14} for an  illustration). This concept was developed later by Moffatt \cite{Mo85}. The magnetic relaxation procedure envisioned by Moffatt formally preserves the streamline topology of an initial divergence free three-dimensional magnetic field,    whose evolution under the magnetic relaxation equations  is conjectured to converge in the infinite time limit towards ideal Euler equilibria.    There are other magnetic relaxation equations for which the steady states are incompressible Euler equilibria. We mention for instance the models of  Vallis-Carnevale-Young\cite{VCY89} and Nishiyama \cite{Ni03}.    
For more details  about topological aspects of fluid dynamics,   see Moffatt's recent overview  \cite{Mo21}.
\medskip

In this paper, we consider the Cauchy problem of a magnetic relaxation equations (MRE) considered in Beekie-Friedlander-Vicol  \cite{BFV} (which is motivated by Moffatt \cite{Mo85}) on the flat torus $\T^d=\R^d/\Z^d$ with $d\geq 2:$
\begin{equation}\label{MRE}
\left\{\begin{aligned}
 \partial_t B+u\cdot\nabla B  &=B\cdot\nabla u,\\
 (-\Delta)^\gamma u&=B\cdot\nabla B+\nabla P,\\
 \div u=\div B&=0,
\end{aligned}\right.
\end{equation}
supplemented with the initial conditions
\begin{equation}\label{eq-indata}
  B(0, x)=B_0(x),\quad x\in \T^d.
\end{equation}
The unknowns are the  incompressible magnetic vector field $B,$ the   incompressible velocity vector field $u$  which is taken to have zero mean on $\T^d$, and  the scalar pressure $P.$
The  parameter $\gamma\geq 0$ is a regularization parameter of the constitutive law\footnote{See \eqref{Def-fracL-P} for a definition of fractional Laplacian on $\T^d$.} $B\mapsto u:$
\begin{equation}\label{C-law}
    u=(-\Delta)^{-\gamma}\mathbb{P}(B\cdot\nabla B)=(-\Delta)^{-\gamma}\mathbb{P}\div(B\otimes B),
\end{equation}
  where $\mathbb{P}$ is the Leray projector  onto divergence free vector fields.
\bigbreak

First and foremost, let us observe from \eqref{MRE}$_3$ 
 that for any smooth solution $(B, u)$ of \eqref{MRE} we have
\begin{equation}\label{en-eq}
    \frac{d}{dt}\|B(t, \cdot)\|_{L^2}^2+  \|u(t, \cdot)\|_{\dot{H}^\gamma}^2 +\|\mathbb{P}\div(B\otimes B)\|_{\dot{H}^{-\gamma}}^2= 0.
\end{equation}
From   above energy inequality, it seems proper to  define a distribution $u$ via  \eqref{MRE}$_2.$   
Note that in order to define weak solutions for \eqref{MRE}$_1$ a minimal requirement for $u(t, \cdot)$ is the square integrability,  when $B(t,\cdot)$ is square integrable. By the Sobolev embedding and \eqref{MRE}$_2$, this   would require $\gamma>{d}/{4}+\frac{1}{2}.$  The energy  equality \eqref{en-eq} comes to the rescue, it provides for any $\gamma\geq0$ the required square integrability for both $u$ and $B$ in space.  However, the existence of global weak solutions  for \eqref{MRE} supplemented with rough initial data (e.g. merely bounded data) is   challenging even when $\gamma>{d}/{4}+\frac{1}{2},$  since energy inequality \eqref{en-eq} does not yield robust compactness properties for the magnetic field to pass to the limit in the nonlinear term $B\cdot\nabla B$ in \eqref{MRE}$_2$.  To our best knowledge, the only existence result with rough initial data is due to  Brenier \cite{Br14}.  In the two dimensional case and with $\gamma=0,$ he obtained  global-in-time dissipative weak solutions  in the spirit of  Lions' dissipative weak solutions to the Euler equations \cite{PLL}, a notion of solution which is weaker than the one of an usual weak solution, but which holds a weak-strong uniqueness property.    On the other hand, let us mention    Constantin and Pasqualotto \cite{CP23}  recently constructed   Euler/magnetostatic equilibrium from certain Voigt regularizations of the classical incompressible MHD system without resistivity.  The regularization considered in \cite{CP23} gives additional compactness for the magnetic field, which allows to  pass to the limit in the  expression $B\cdot\nabla B,$ and obtain global-in-time solutions.  However, the price for that  Voigt type regularization is poor control on the topology of the limiting  magnetic field ($t\to\infty$), because  the topology of the magnetic field  is not  preserved along time evolution.   We also mention the work \cite{KK23} of Kim and Kwon on the global existence of weak solutions of  Stokes-Magneto system with   fractional diffusions.   Later,  in the case without magnetic resistivity,  Bae-Kwon-Shin \cite{BKS23} established the local and global well-posedness for smooth enough initial magnetic field.   Local existence and uniqueness of solutions in low regularity space for other   magnetic relaxation equations were proved in \cite{F14, C16, F17, Li17}.
  
  \bigbreak
  
At this stage, let us recall the results obtained for \eqref{MRE} in Beekie-Friedlander-Vicol \cite{BFV}. Local-in-time  existence and uniqueness of solutions were established   for initial data $B_0$ in Sobolev spaces $H^s$ with $s>{d}/{2}+1$, for any $\gamma\geq 0.$ Global well-posedness result was obtained if further $\gamma>{d}/{2}+1.$  The question of global regularity or finite time blowup for \eqref{MRE} when $\gamma\in[0, {d}/{2}+1]$ was listed as an open 
problem in \cite{BFV}.  About Moffatt's conjecture on magnetic relaxation for \eqref{MRE},  it is showed that there exists a solution such that  the Lipschitz norm of $u$ is not integrable in the whole time interval $\R_+.$ 
More precisely, motivated by the work of Elgindi and Masmoudi \cite{EM20} on the Euler equations, in three-dimensional case they exhibited an example of  current $\nabla\times B$ that  grows exponentially in time.
These examples suggest the relaxation towards Euler  equilibrium from \eqref{MRE} is a subtle matter, 
as for general initial data one cannot expect magnetic relaxation with respect to strong norms.
This fact motivates them  established in \cite[Theorem 5.1]{BFV} the asymptotic stability of a special Euler equilibrium $e_1:=(1, 0)^{\rm T}$ under Sobolev smooth perturbations in the two-dimensional case and $\gamma=0$. 
\bigbreak

The main goal of this paper is to establish the  existence, uniqueness and regularity property of global weak solutions for \eqref{MRE} subject to (almost)  finite-energy  initial data when the regularization parameter $\gamma\geq {d}/{2}+1.$  Before stating our main results, let us first give the following definition of weak solutions for Moffatt's MRE \eqref{MRE} with general $\gamma\geq0.$

\begin{definition}\label{Def-ws}
Let $\gamma\geq 0.$  Let initial data $B_0\in L^2(\T^d)$ and   $\div B_0=0.$ A pair of vector fields $(B, u)(t, x)$   is called a global weak solution to  the Cauchy problem \eqref{MRE}-\eqref{eq-indata},   if $(B, u)\in L^\infty(0, T;  L^2(\T^d))\times L^2(0, T; \dot{H}^\gamma(\T^d))$  for all $0<T<\infty,$  and   it satisfies for almost every  $0<t<T,$
\begin{multline*}
    \int_0^t \int_{\T^d} [B\cdot\partial_t v+ (B\otimes u):\nabla v-(u\otimes B):\nabla v]\,dxd\tau\\
   =\int_{\T^d}B(t, \cdot)\cdot v(t, \cdot)\,dx- \int_{\T^d}B_0\cdot v(0, \cdot)\,dx,
\end{multline*}
for any vector test  function $v\in \mathcal{C}^\infty([0, T)\times \T^d),$ and 
\begin{equation*}
    \int_{\T^d} [u(t, \cdot)\cdot(-\Delta)^\gamma w+ (B(t, \cdot)\otimes B(t, \cdot)):\nabla w]\,dx=0,
\end{equation*}
 for any vector test  function $w\in\mathcal{C}^\infty(\T^d)$ with $\div w=0,$ and  
 \begin{equation*}
     \int_{\T^d} B(t, \cdot)\cdot\nabla \phi\,dx= \int_{\T^d} u(t, \cdot)\cdot\nabla \psi\,dx=0,
\end{equation*}
 for any scalar test functions   $\phi, \psi\in \mathcal{C}^\infty(\T^d).$ 

Moreover, the   energy inequality holds, namely,  
\begin{equation}\label{en-ineq-def}
      \|B(t, \cdot)\|_{L^2 }^2+2\int_{0}^t\|u(\tau, \cdot)\|_{\dot{H}^{\gamma} }^2\,d\tau\leq \|B_0\|_{L^2 }^2.
\end{equation}
\end{definition}
\bigbreak

\subsection{Main results} 
Now, we are ready to state our main results. The first theorem shows the existence  and uniqueness  of   weak solutions  in the sense of Definition \ref{Def-ws}  for merely finite-energy initial data, whenever the regularization parameter satisfying $\gamma> \gamma_c:=d/2+1.$  
\begin{theorem}\label{Th2-1}
    Let $\gamma>\gamma_c.$ Given any initial data $B_0\in L^2(\T^d)$ and $\div B_0=0,$ then the  Cauchy problem \eqref{MRE}--\eqref{eq-indata} admits a unique global-in-time weak solution $(B, u)\in (L^\infty(\R_+; L^2)\cap \mathcal{C}_{\rm loc}(\R_+; L^2))\times  L^\infty(\R_+;  H^\gamma ),$ which makes the energy inequality  \eqref{en-ineq-def} becomes an equality.
\end{theorem}

In the next, we show the existence  and uniqueness  of   weak solutions  when the regularization parameter reaches the borderline case $\gamma=\gamma_c.$  To do that, we need the following global well-posedness  result for Moffatt's MRE \eqref{MRE}  supplemented with smooth initial data  and with borderline regularization parameter $\gamma={d}/{2}+1,$  which reads 

\begin{proposition}\label{Th1}
    Let $\gamma=\gamma_c.$ Given any initial data $B_0\in H^s(\T^d)$ with $s>{d}/{2}+1$ and $\div B_0=0.$ Then for any $T>0,$ there exists a  unique  solution $B\in \mathcal{C}([0,  T]; H^s)$ with associated velocity $u\in \mathcal{C}([0, T]; H^{s-1+2\gamma_c} )\cap L^2(0, T; H^{s+\gamma_c} ).$  Moreover, the pair $(B, u)$ satisfies the following energy equality: 
    \begin{equation}\label{G-en-eq}
        \|B(t, \cdot)\|_{L^2 }^2+2\int_{t_0}^t\|u(\tau, \cdot)\|_{\dot{H}^{\gamma_c} }^2\,d\tau=\|B(t_0, \cdot)\|_{L^2 }^2,
    \end{equation}
   for all $0\leq t_0\leq t\leq T,$ and also the bound:
\begin{multline}\label{Th1-bound}
     \|B(t,\cdot)\|_{ {H}^s }^2+\int_0^t\|u(\tau, \cdot)\|_{{H}^{s+\gamma_c} }^2\,d\tau\\
     \leq  \|B_0\|_{{H}^s }^2 \,e^{e^{C \sqrt{t}\|B_0\|_{L^2}(1+t\|B_0\|_{H^s}^2)   (Ct+\| B_0\|_{L^p}^2)\exp(C\sqrt{t}\|B_0\|_{L^2})\,}},
\end{multline}
where $C$ is a positive constant depends only on $s$ and  $d.$
\end{proposition}

Then, we have

\begin{theorem}\label{Th2}
    Let $\gamma= \gamma_c$ and $\eps>0.$ Given any initial data  $B_0\in L^{2+\eps}(\T^d)$ and $\div B_0=0,$    then the Cauchy problem \eqref{MRE}--\eqref{eq-indata}  admits   a unique  global-in-time weak solution $(B, u)\in \left(L^\infty(\R_+; L^2)\cap \mathcal{C}_{\rm loc}(\R_+; L^2)\cap L^\infty_{\rm loc}(\R_+; L^{2+\eps})\right)\times  L^\infty_{\rm loc}(\R_+; W^{1, \infty}),$ which  satisfies the energy equality  \eqref{G-en-eq}  and the bound:  
\begin{align}\label{es-Th2}
    ~~~ \| B(t, \cdot)\|_{L^{2+\eps}}^2\leq (e+\| B_0\|_{L^{2+\epsilon}}^2)^{\exp(C\sqrt{t}\|B_0\|_{L^2})}\,   e^{Ct\,\exp(C\sqrt{t}\|B_0\|_{L^2})},
\end{align}
for almost every $t\in[0, \infty),$ where $C>0$  depends only on  $d$ and $\eps.$ 
\end{theorem}
 Note that the bound \eqref{es-Th2} is important as it implies Lipschitz regularity of the velocity field,  while \eqref{G-en-eq} does not.
\bigbreak

In the last, we investigate  Moffatt's magnetic relaxation procedure.   
This is a subtle matter that shown by  Beekie-Friedlander-Vicol's examples of growth of the current.   Here, we discuss the possible asymptotic behavior of the solutions obtained in Theorem \ref{Th2-1} and Theorem \ref{Th2}, i.e. for large values of $\gamma.$  The result below shows the relaxation of the  velocity field of these solutions as $t\to\infty.$ 
\begin{corollary}\label{C0}
    Let $(B, u)$ be the weak solutions obtained in Theorem \ref{Th2-1}, then  
 \begin{equation}\label{behavior-u1}
        \lim_{t\to\infty}\|u(t, \cdot)\|_{{H}^\alpha}=0, \quad {\rm for~all}~0\leq \alpha<2\gamma-\gamma_c.
    \end{equation}
    Also,  for  solutions obtained in Theorem \ref{Th2},  one has
 \begin{equation}\label{behavior-u2}
        \lim_{t\to\infty}\|u(t, \cdot)\|_{{H}^{\beta}}=0,\quad {\rm for~all}~0\leq \beta<\gamma_c.
    \end{equation}
\end{corollary}
However,   the above corollary is not enough to  prove that the   magnetic field   $B(t, \cdot)$  will relax to a steady state $B^{\infty}$ of the Euler equations \eqref{eq-Euler} as $t\to\infty.$    Indeed,  the constitutive law \eqref{C-law} endows a cubic nature of the nonlinear term in the magnetic equation \eqref{MRE}$_1$, which may lead to a growth  of the magnetic field even in weak norms  as suggested in \eqref{es-Th2}.   As such,  for any $\gamma,$   it remains challenging to  discover the nonlinear relaxation mechanism of Moffatt's MRE. 
In what follows, we attempt to identify   the limiting profiles of Moffatt's MRE  for initial data perturbed around a uniform magnetic field.  which is motivated by \cite{DGL}.    Let us recall a result first obtained in  \cite{BFV}.
\begin{theorem}[\cite{BFV}]\label{Th-BFV}
  Let $d=2.$  Let $k\geq 4$ and $m\geq k+9.$ Choose $\delta\in(0, 1).$ There exists a small enough number $\eta_0$ such that for any solenoidal vector field $B_0$ if $\|B_0-e_1\|_{H^m}=\eta\leq \eta_0,$
  and $\mathbb{P}_0(B_0-e_1)=0,$
    then we have that \eqref{MRE}-\eqref{eq-indata} has a unique global in time solution $(B, u),$ which satisfies $\|B-e_1\|_{L^2}\leq \eta,$ $\mathbb{P}_0 B^2\equiv0$ and 
   \begin{align}
       \|\mathbb{P}_{\perp} (B-e_1)\|_{\dot{H}^k}&\leq 4\eta \,e^{-(1-\delta)t},\label{ThBFV1}\\
         \|\mathbb{P}_{0}(B^1-1)\|_{{H}^{k+2}}&\leq 4\eta, \label{ThBFV2} \\
           \|B-e_1\|_{\dot{H}^m}^2&\leq 4\eta \,e^{\eta t}
   \end{align}
   for $t\in[0, \infty).$ In above,  for any function $\psi: \T^2\to \R,$ we used the notations
   \begin{align*}
       (\mathbb{P}_0 \psi)(x_2):=&\int_{0}^1\psi(x_1, x_2)\,dx_1,\\
       (\mathbb{P_\perp}) \psi(x_1, x_2):=& \psi(x_1, x_2)- (\mathbb{P}_0 \psi)(x_2).
   \end{align*}
   As a consequence, the  velocity field satisfies $u(t, \cdot)\to 0$ as $t\to \infty,$ whereas the magnetic field $B(t, \cdot)$ relaxes to a steady state $B^{\infty}$ with $\|B^\infty-e_1\|_{H^{k+2}}\leq 4\eta,$ both convergences taking place with respect to strong topologies.
\end{theorem}
Then, we  show that
\begin{corollary}\label{C1}
   The steady state $B^\infty$ obtained in Theorem \ref{Th-BFV} is a vector-valued function of the single variable $x_2\in\T$ (that corresponds to a shear flow of the Euler equations)  given by\footnote{We use dot for differentiation with spatial variable $x_2.$}   $$B^\infty(x_1,  x_2)= (-  (\dot{\phi}^\infty)(x_2), 0)$$
   with
  \begin{align}\label{formula-phi}
     \phi^\infty(s):= \int_{0}^\infty {\bf 1}_{ \{|\{|\phi_0|>\lambda\}|\geq s\}}(\lambda)\,d\,\lambda\quad{\rm for}~~ s\in[0, 1],
   \end{align}
   where $\phi_0: \T\times [0, 1]\to \R$ is the stream function of  $B_0$ such that $(\mathbb{P}_0 \dot\phi_0)(\cdot)=-1.$ 
\end{corollary}  
\begin{proof}
 In the two-dimensional case,  because  $\div B=0$  we may identify a (unique up to an additive constant) scalar magnetic stream function $\phi$ on $\R_+\times( \T\times [0, 1]),$  which satisfies the active scalar equation 
   \begin{equation*}
   \left\{\begin{aligned}
         \partial_t \phi+u\cdot\nabla \phi=0\\
         \phi(t,  \cdot)|_{t=0}=\phi_0(\cdot).
   \end{aligned}\right.
   \end{equation*}
Note that  $ \nabla^{\perp} \phi =B,$ where $\nabla^{\perp}:=(-\partial_2, \partial_1),$ thus in the current framework $\phi$ cannot  be a periodic function in terms of the vertical spatial variable.

Thanks to  \eqref{ThBFV1} and the fact that $\mathbb{P}_0 B^2\equiv 0,$   we see that
  \begin{align*}
      \|\partial_1 \phi(t, \cdot)\|_{\dot{H}^k}=  \|\mathbb{P}_{\perp} B^2\|_{\dot{H}^k}\leq 4\eta \,e^{-(1-\delta)t}\to 0, \quad{\rm as}~~~\, t\to\infty.
  \end{align*}
Putting this together with \eqref{ThBFV2} thus yields
  $\phi(t, \cdot)\to \phi^*$ in strong topologies  as $t\to \infty,$ for some $\phi^*$ that depending only on the vertical spatial variable and  satisfying $\dot{\phi^*}<0.$
  In other words, \emph{$\phi^*$ is a vertical decreasing   rearrangement\footnote{We say that two maps $\phi_1, \phi_2:\T^d\to \R$ are rearrangement of each other if $|\{\phi_1>\lambda\}|=|\{\phi_2>\lambda\}|$ for all $\lambda\geq0.$} of $\phi_0.$} The classical rearrangement theory in e.g.   \cite[Chapter 3]{LL01} tells us that  $\phi^\infty$ given in \eqref{formula-phi}  is the unique vertical decreasing   rearrangement of $\phi_0,$ therefore $\phi^*=\phi^\infty.$
  \end{proof}
\bigbreak

\subsection{A few remarks on the main results} 

First and foremost,   Theorem \ref{Th2-1}   partially answers an open problem  that proposed in \cite[Page 1337, Q4]{BFV}, which concerns the existence of weak solutions for any $\gamma\geq0$ and the validity of the energy inequality \eqref{en-ineq-def} for $\gamma\in(d/4+1/2, d/2+1].$  As a consequence of   our result, the existence of weak solutions for $\gamma\in[0, d/2+1]$ remains open.   
 We note that the essential difficulty  for proving existence of a weak solution for MRE \eqref{MRE} lies in the lack of  robust compactness property of the magnetic field, which can not be directly obtained  from energy inequality \eqref{en-eq},  for any $\gamma\geq0.$ To  obtain strong convergence property of the magnetic field, we establish a key compactness result for the following linear active vector equations:
\begin{equation}
\left\{
\begin{aligned}
&\partial_t A_n+\div(A_n \otimes v_n)=\div(v_n\otimes A_n),\\
&\div A_n=\div{v}_n=0.
\end{aligned}
\right.
\end{equation}
The result  stated in Lemma \ref{Lemma3.1} can be understood as a generalization of  Lions' fundamental compactness result \cite{PLL}  for scalar transport equation.     
\medskip

Secondly, Proposition \ref{Th1} answers a borderline case of  an open problem  that was proposed in \cite[Page 1336, Q1]{BFV}.   Notice that the global regularity problem of  the borderline case $\gamma=\gamma_c$ is not accessible by the methods of \cite{BFV}. Indeed, in this case, the energy inequality \eqref{en-eq} does not ensure the  velocity field is Lipschitz continuous but only log-Lipschitz continuous.  To deal with this borderline situation,  we take use of some logarithmic interpolation inequalities  and  observe that the $L^{2+\eps}$ norm of the magnetic field can be controlled by its initial norm through a double exponential inequality, see \eqref{es-Th2}. 
As a consequence, we obtain   global-in-time bound of  Lipschitz norm of the velocity field in  \eqref{es-2.2-000}. 
 Still, these new bounds cannot be applied to  previous known regularity criteria.  We thus  provide a refined Beale-Kato-Majda type regularity criterion,  i.e.  Proposition \ref{Prop-criteria}.   Furthermore,  Theorem \ref{Th2} shows that for  initial data only in  $L^{2+\eps}$,  the global-in-time existence and uniqueness result remains valid.
\medskip

Our analysis relies heavily  on the Littlewood-Paley decomposition   and Besov spaces (see Appendix \ref{A} for definitions),    as this framework is well-suited to studying the well-posedness of rough solutions for Moffatt's MRE \eqref{MRE}. Indeed,  noticing that the nonlinear term $B\otimes B$ may only be in $L^1,$  and standard $L^1$ estimates fail for the Stokes-type equation governing the velocity field because Calder\'{o}n-Zygmund operators are unbounded in $L^1$ (and $L^\infty$).    For uniqueness,  the active vector equation for the magnetic field inherently causes a loss of one derivative in the estimates for the difference of two solutions.  More precisely,  even for large $\gamma,$  it seems difficult to obtain an $L^2$ estimate for the equation
\begin{equation*}
\partial_t \delta B + u_1\cdot\nabla \delta B=B_1\cdot\nabla \delta u+\delta B\cdot\nabla u_2- \delta u\cdot\nabla B_2,
\end{equation*}
which is satisfied by $\delta B= B_1-B_2,$ where $B_1$ and $B_2$ are two solutions emanating from the same initial datum.  The troublemaker is the term $\delta u\cdot\nabla B_2$,  due to the  lack of regularity of solutions.   As such  we handle the uniqueness issue  by the Lagrangian approach, and it relies on some estimates in Besov spaces (instead of $L^1$) for related  Stokes system.
Let us mention that  after transforming the problem in Lagrangian coordinates,  the difference analysis is  not standard for the case $\gamma\notin \N.$ Indeed, since the fractional Laplacian operator is {\em{nonlocal}},  we have to  formulate  an associated operator in the Lagrangian coordinates, and provide some necessary estimates for such operators.

\medskip
From a mathematical perspective, the analysis of the MRE \eqref{MRE} is  both interesting and challenging.   On the one hand,   the combination of the constitutive law \eqref{C-law} and the energy inequality \eqref{en-ineq-def} suggests the existence of a relaxation mechanism in the magnetic field that has yet to be discovered.  Regarding this topic,  we establish in Corollary \ref{C0} the relaxation of the velocity field for the case that $\gamma\geq \gamma_c$. 
On the other hand,  MRE \eqref{MRE} shares some similarities with the  incompressible Euler equations \eqref{eq-Euler} in vorticity formulation, and its solutions are conjectured   to converge asymptotically  to the steady states of \eqref{eq-Euler} as time diverges.     In this regard,  Corollary \ref{C1} seems to be  the first result that provides non-trivial explicit limiting profiles in the context of the magnetic relaxation problem.
\medskip

  After the completion of the current manuscript (arXiv:2311.18407),  the author  learned  that  a global well-posedness result for the borderline case $\gamma=\gamma_c$,  similar to Proposition \ref{Th1},     had been obtained  earlier  in Bae-Kwon-Shin \cite[Theorem 2.2]{BKS23} (arXiv:2310.03255).        The work of Bae–Kwon–Shin  primarily focuses on  the case  of  a magnetic field endowed with  fractional diffusion,   which   constitutes a fundamental difference from the present work.  They  also address the case without magnetic diffusion and  improve upon  the local well-posedness result obtained  by Beekie-Friedlander-Vicol \cite{BFV}.       
  Compared to their result,  the bounds \eqref{Th1-bound}, \eqref{es-2.2-111}, \eqref{es-2.2-000} and Proposition \ref{Prop-criteria}  are novel.    More importantly,   Proposition \ref{Th1}   serves here  only as an  intermediate step toward establishing  the global existence of almost  finite-energy solutions,   a result that is non-trivial.

\subsection*{Notation} 
We end this introductory part by presenting  a few notations. 
We denote by $C$ harmless positive `constants' that  may change from one line to the other, and we sometimes write $A\lesssim B$ instead of $A\leq C B.$  Likewise,    $A\sim B$ means that  $C_1 B\leq A\leq C_2 B$ with absolute constants $C_1$, $C_2$. 
Throughout the paper, $i$-th coordinate of a vector $v$ will be denoted  
by $v^i.$  For a real-valued matrix ${\rm M}$, ${\rm M}^{\rm T}$ represents its transpose, while for two multidimensional real-valued matrices ${\rm M}_1, {\rm M}_2,$  ${\rm M}_1:{\rm M}_2$ denotes their standard inner product. We use the notation $\Lambda:=(-\Delta)^{{1}/{2}},$  $\left\langle~,~ \right\rangle$ for the $L^2$ inner product  and $[\mathbb{A}, \mathbb{B}]:=\mathbb{A}\mathbb{B}-\mathbb{B}\mathbb{A}$  for
the commutator of two operators.   
For $\mathcal{X}$ a Banach space, $p\in[1, \infty]$ and $T\in(0,\infty]$, the notation $L^p(0, T; \mathcal{X})$  designates the set of measurable functions $f: [0, T]\to \mathcal{X}$ with $t\mapsto\|f(t)\|_\mathcal{X}$ in $L^p(0, T)$, endowed with the norm $\|\|\cdot\|_\mathcal{X}\|_{L^p(0, T)}.$ For any interval $I$ of $\R,$ we agree that $\mathcal{C}(I; \mathcal{X})$  denotes the set of continuous 
  functions from $I$ to $\mathcal{X}$.  We  keep the same notation for functions with several components.
\medskip
 
\subsection*{Structure of the paper} 
{ The rest of the paper unfolds as follows.} In the next section, we focus on the proof of  Proposition \ref{Th1}, which is fundamental to the proof of Theorem \ref{Th2}.   Section \ref{S3} is devoted to the proof of existence in  Theorem \ref{Th2-1} and   \ref{Th2},  via an unified approach.       In Section \ref{S4},  we prove the uniqueness part which will be divided into the cases $\gamma\in\N$ and $\gamma\notin\N.$          The proof of Corollary \ref{C0} is placed  last.   For the reader's convenience,  results concerning Besov spaces and  Littlewood–Paley decomposition on $\T^d$ are recalled in Appendix \ref{A}.

 \section{Global regularity for the borderline case \texorpdfstring{$\gamma={d}/{2}+1$}{TEXT}}
 This section is devoted to the proof of Proposition  \ref{Th1}. To do that, at first we need the following Beale-Kato-Majda type regularity criteria.
 
 \subsection{A refined regularity criteria for $\gamma>d/2$} 
 \begin{proposition}\label{Prop-criteria}
     Let $\gamma>{d}/{2}, s>{d}/{2}+1.$ Assume that the initial magnetic field $B_0\in H^s $ and $\div B_0=0.$ Then, there exists a time $T^*=T^*(\|B_0\|_{H^s}),$ such that the problem  \eqref{MRE}--\eqref{eq-indata} has  
     a unique   solution $B\in \mathcal{C}([0, T^*); H^s )$ with associated velocity $u\in \mathcal{C}([0, T^*); H^{s-1+2\gamma} )\cap L^2(0, T^*; H^{s+\gamma} ).$  Moreover, it holds that  
\begin{equation}\label{blowup-criterion}
     \|B(t,\cdot)\|_{{H}^s }^2\leq \|B_0\|_{{H}^s}^2 \,e^{e^{C (1+t\|B_0\|_{H^s}^2 ){e^{CV(t)}}}},
\end{equation}
for $t\in [0, T^*),$  where $C$ is a positive constant depending only on $\gamma, s, d,$ and  $$V(t):=\int_0^t (\|\nabla u(\tau, \cdot)\|_{L^\infty}+\|\Lambda u(\tau, \cdot)\|_{L^\infty})\,d\tau.$$
 \end{proposition}
 \begin{proof}
     Thanks to \cite[Theorem 2.2]{BFV},  we only need to prove the bound \eqref{Prop-criteria}.  Applying the operator $\Lambda^r$  $(r>0)$ to both sides of MRE \eqref{MRE} and testing the resulting equations with $\Lambda^r B, ~\Lambda^r u,$ respectively, we then obtain that
     \begin{multline}\label{es-2.1-1}
         \frac{1}{2}\frac{d}{dt}\|B(t, \cdot)\|_{\dot{H}^r}^2+\|u(t, \cdot)\|_{\dot{H}^{r+\gamma}}^2\\
         = \langle [\Lambda^r, B\cdot\nabla ]u,  \Lambda^r B \rangle-\langle[\Lambda^r, u\cdot\nabla ]B, \Lambda^r B\rangle+\langle [\Lambda^r, B\cdot\nabla ]B, \Lambda^r u\rangle.
     \end{multline}
     Recalling the following Kato-Ponce type commutator estimates from Lemma \ref{Le-KP}:
for all $r>0,$
      \begin{equation}\label{ineq-Li1}
         \|[\Lambda^r f, g]\|_{L^2}\lesssim \|\Lambda^r f\|_{L^2}\|g\|_{L^\infty}+\|\nabla f\|_{L^2}\|\Lambda^{r-1}g\|_{L^\infty}
     \end{equation}
     and
     \begin{equation}\label{ineq-Li2}
         \|[\Lambda^r f, g]\|_{L^2}\lesssim \|\Lambda^r f\|_{L^\infty}\|g\|_{L^2}+\|\nabla f\|_{L^\infty}\|\Lambda^{r-1}g\|_{L^2}.
     \end{equation}
Now, use    \eqref{ineq-Li1}  and the Sobolev embedding $H^\gamma\hookrightarrow L^\infty$  (since $\gamma>{d}/{2}),$ we have
\begin{align}\label{es-2.1-2}
    |\langle [\Lambda^r, B\cdot\nabla ]u,  \Lambda^r B \rangle|&\lesssim \|[\Lambda^r, B\cdot\nabla ]u\|_{L^2 } \|\Lambda^r B\|_{L^2}\notag\\
    &\lesssim (\|\Lambda^r B\|_{L^2}\|\nabla u\|_{L^\infty}+\|\nabla B\|_{L^2}\|\Lambda^{r-1}\nabla u\|_{L^\infty})\|B\|_{\dot{H}^r}\notag\\
    &\lesssim (\|  B\|_{\dot{H}^r}\|\nabla u\|_{L^\infty}+\|\nabla B\|_{L^2}\| u\|_{H^{r+\gamma}})\|B\|_{\dot{H}^r}\cdotp
\end{align}
Similarly, \eqref{ineq-Li2} yields that
\begin{align}\label{es-2.1-3}
    |\langle [\Lambda^r, u\cdot\nabla ]B,  \Lambda^r B \rangle|&\lesssim \|[\Lambda^r, u\cdot\nabla ]B\|_{L^2 } \|\Lambda^r B\|_{L^2}\notag\\
    &\lesssim (\|\Lambda^r u\|_{L^\infty}\|\nabla B\|_{L^2}+\|\nabla u\|_{L^\infty}\|\Lambda^{r-1}\nabla B\|_{L^2})\|B\|_{\dot{H}^r}\notag\\
    &\lesssim (\|   u\|_{{H}^{r+\gamma}}\|\nabla B\|_{L^2}+\|\nabla u\|_{L^\infty}\| B\|_{\dot{H}^{r }})\|B\|_{\dot{H}^r}\cdotp
\end{align}
   We  decompose  the third term in \eqref{es-2.1-1}, by the fact that $\div B=0,$ into
\begin{align*} 
      \langle [\Lambda^r, B\cdot\nabla ]B, \Lambda^r u\rangle= -\langle \Lambda^{r}(B\otimes B), \nabla\Lambda^{r} u\rangle+ \langle B\cdot\nabla \Lambda^r u, \Lambda^r B \rangle.
\end{align*}
Then, Lemma \ref{Le-KP2} implies that  
\begin{align*}
   | \langle \Lambda^{r}(B\otimes B), \nabla\Lambda^{r} u\rangle|&\lesssim \|B\otimes B\|_{{\dot{H}}^{r}} \|u\|_{\dot{H}^{r+1}} \\
   &\lesssim \|B \|_{{H}^{r}} \|B\|_{L^\infty}\|u\|_{\dot{H}^{r+1}},
\end{align*}
and H\"{o}lder's inequality implies that
\begin{align*}
   | \langle B\cdot\nabla \Lambda^r u, \Lambda^r B \rangle|&\lesssim \|B\|_{L^\infty}\|\nabla \Lambda^r u\|_{L^2} \|\Lambda^r B\|_{L^2}\\
   &\lesssim \|B\|_{L^\infty}\|  u\|_{\dot{H}^{r+1}} \|B \|_{{\dot{H}}^{r}}.
\end{align*}
One finds from above two inequalities   and $\gamma>{d}/{2}\geq 1$ that
\begin{align}\label{es-2.1-5}
    | \langle [\Lambda^r, B\cdot\nabla ]B, \Lambda^r u\rangle|\lesssim   \|B\|_{L^\infty}\|  B\|_{{H}^r}  \|  u\|_{H^{r+\gamma}}.
\end{align}
Putting inequalities \eqref{es-2.1-2}-\eqref{es-2.1-5} and energy inequality \eqref{en-eq}
into inequality \eqref{es-2.1-1}, and using Young's inequality  and the fact that $u$ has zero mean, it gives that
\begin{equation}\label{es-2.1-6}
\frac{d}{dt}\|B \|_{ {H}^r}^2+\|u \|_{ {H}^{r+\gamma}}^2\leq C \|B\|_{ {H}^r}^2(\|\nabla u\|_{L^\infty}+\|\nabla B\|_{L^2}^2 +\|B\|_{L^\infty}^2).
\end{equation}

At this stage, we have to estimate  $\|\nabla B\|_{L^2}$ and  $\|B\|_{L^\infty}.$  In fact, in the case $r=1,$ from \eqref{es-2.1-5} and the last two inequality in \eqref{es-2.1-2} and \eqref{es-2.1-3},   we can improve  \eqref{es-2.1-6}  as
\begin{equation}\label{es-2.1-7}
    \frac{d}{dt}\|B\|_{ {H}^1}^2+\|u\|_{ {H}^{1+\gamma}}^2\lesssim \|B\|_{ {H}^1}^2(\|\nabla u\|_{L^\infty}+\|\Lambda u\|_{L^\infty}+  \|B\|_{L^\infty}^2).
\end{equation}
Applying the maximal principal for magnetic equation \eqref{MRE}$_1$  yields that
\begin{align}  
\| B(t, \cdot)\|_{L^\infty}&\leq   \|  B_0\|_{L^\infty}  \,\exp\left(C\int_0^t \|\nabla u(\tau, \cdot)\|_{L^\infty}\,d\tau\right)\notag\\
& \leq \|  B_0\|_{L^\infty}  \,\exp(CV(t)),\label{es-2.1-00}
\end{align}
   which together with \eqref{es-2.1-7} imply that
\begin{multline*}
    \| B(t, \cdot)\|_{H^1}^2+\|B(t, \cdot)\|_{L^\infty}^2\\
    \leq (\|  B_0\|_{H^1}^2+\|B_0\|_{L^\infty}^2)\exp\{C(1+t\|B_0\|_{L^\infty}^2)\exp(CV(t))\}.
\end{multline*}
Finally, taking use of the above inequality and \eqref{es-2.1-6} we get for any $r>0,$ 
\begin{equation}\label{es-Sobolev}
    \|B(t, \cdot)\|_{H^r}^2\leq   \|B_0\|_{H^r}^2 \exp\left\{\exp\left\{C(1+tC_0 ) \exp(C V(t))\right\}\right\},
\end{equation}
with $C_0=\|B_0\|_{H^1}^2+\|B_0\|_{L^\infty}^2.$

Since $s>{d}/{2},$ one has $C_0\lesssim \|B_0\|_{H^s}^2.$ Thus 
 by taking $r=s$ in \eqref{es-Sobolev}   we  complete the proof of Proposition \ref{Prop-criteria}.
\end{proof}

 \subsection{The global regularity}
 Now, we are ready to prove Proposition \ref{Th1}. Thanks to Proposition \ref{Prop-criteria},  we only need to obtain the boundness of $V(t)$ for any $t>0.$  But, in the case $\gamma=\gamma_c$ energy inequality \eqref{en-eq} does not give the bound directly, since the failure of the Sobolev embedding $H^{\gamma_c-1}\hookrightarrow L^\infty.$ To deal with  this difficulty, we show global-in-time bound of the magnetic field in $L^p$  for all $p>2$, which in turn gives the desired estimate of $u$ thanks to the constitutive law \eqref{C-law}.
 \begin{proposition}\label{Prop-gc}
    Let $\gamma=\gamma_c.$ Let $(B, u)$ be a smooth solution for the MRE \eqref{MRE}, then for all $p\in(2, \infty],$ it holds that  
\begin{align}\label{es-2.2-111}
     \| B(t, \cdot)\|_{L^p}^2\leq (e+\| B_0\|_{L^p}^2)^{\exp(C\sqrt{t}\|B_0\|_{L^2})}\, \exp\{Ct\,\exp(C\sqrt{t}\|B_0\|_{L^2})\},
\end{align}
and
  \begin{equation}\label{es-2.2-000}
    \int_0^t  \|(\nabla u, \Lambda u)(\tau, \cdot)\|_{L^\infty}\,d\tau\leq C\sqrt{t}\|B_0\|_{L^2}  (Ct+\| B_0\|_{L^p}^2)\exp(C\sqrt{t}\|B_0\|_{L^2}),
\end{equation}
for all $0\leq t<\infty,$  where $C$ is a positive constant depends only on $d$ and $p.$
 \end{proposition}
 
 \begin{proof}\label{Prop-2.2}
Let us first discuss for the case of general $\gamma$ and $p.$  
Taking advantage of the Besov spaces (for definitions, see Appendix \ref{A}),  we  see from the constitute law \eqref{C-law} that
for all $\gamma\geq0$ and $p\in[2, \infty],$  
     \begin{align}
       \| u\|_{\dot{B}^{2\gamma-1-\frac{2d}{p}}_{ \infty, \infty}}   &\lesssim  \| u\|_{\dot{B}^{2\gamma-1}_{ {p}/{2}, \infty}}  \notag\\
        &\lesssim \| \Lambda^{-1}\mathbb{P}\div (B\otimes B)\|_{\dot{B}^0_{ {p}/{2}, \infty} }\notag\\
        &\lesssim \|B\otimes B\|_{L^{{p}/{2}}} 
        \lesssim \|B\|_{L^p}^2, \label{es-2.2-1}
     \end{align}
where we used the properties that $L^{p/2}\hookrightarrow\dot{B}^{0}_{p/2, \infty}\hookrightarrow \dot{B}^{-2d/p}_{\infty, \infty}$ and the operator
$\Lambda^{-1}\mathbb{P}\div$ is bounded on $\dot{B}^{0}_{p/2, \infty}$     from Proposition \ref{P_Besov}.

Whenever $2\gamma-1-{2d}/{p}>1,$   
the logarithmic interpolation inequality from Lemma \ref{lemma-log} and  inequality \eqref{es-2.2-1}  imply that
 \begin{align}
    \|u\|_{\dot{B}^{1 }_{\infty, 1}}&\lesssim   \|u\|_{\dot{B}^{1}_{\infty, \infty}} \ln\left( e+  \frac{\|u\|_{\dot{B}^{1}_{\infty, \infty}}+ \|u\|_{\dot{B}^{2\gamma-1-\frac{2d}{p}}_{\infty, \infty}}}{\|u\|_{\dot{B}^{1}_{\infty, \infty}}}\right)\notag\\
  &\lesssim   1+   \|u\|_{\dot{B}^{1}_{\infty, \infty}} \ln\left( e+ \|u\|_{\dot{B}^{2\gamma-1-\frac{2d}{p}}_{\infty, \infty}}   \right)\notag\\
&\lesssim  1+   \|u\|_{\dot{B}^{1}_{\infty, \infty}} \ln\left( e+  \|B\|_{L^p}^2\right),\label{es-2.1-11}
\end{align}
where the inequality $a\ln\left(e+\frac{a+b}{a}\right)\lesssim 1+a\ln(1+b)$ for $a,   b>0$ was used.
Thus in our setting, i.e. $\gamma=\gamma_c={d}/{2}+1$ with $p>2,$ we know that  inequality \eqref{es-2.1-11} is satisfied. And thanks to the embedding $\dot{H}^{\gamma_c}\hookrightarrow \dot{B}^{1}_{\infty, \infty}$,  \eqref{es-2.1-11} can be further rewritten as
\begin{align}
    \|u\|_{\dot{B}^{1 }_{\infty, 1}}
 \lesssim     1+\|u\|_{\dot{H}^{\gamma_c} }  \log\left( e+   \|B\|_{L^p}^2\right).\label{es-2.2-22}
\end{align}
Taking advantage of  energy inequality \eqref{en-eq}, that is
 \begin{equation}\label{en-eq-gamma_c}
    \int_0^\infty \|u(t, \cdot)\|_{\dot{H}^{\gamma_c}}^2\,dt\leq \|B_0\|_{L^2}^2,
\end{equation}
 it remains to   estimate $\|B\|_{L^p}.$ In fact,
similar to \eqref{es-2.1-00}, by the standard $L^p$-estimate of transport equation, one has 
\begin{align*}  
\| B(t, \cdot)\|_{L^p}^2&\leq   \|  B_0\|_{L^p}^2\,\exp\left(C\int_0^t \|\nabla u(\tau, \cdot)\|_{L^\infty}\,d\tau\right)\\
&\leq   \|  B_0\|_{L^p}^2\,\exp\left(C\int_0^t \|  u(\tau, \cdot)\|_{\dot{B}^{1}_{\infty, 1}}\,d\tau\right).
\end{align*}
  This combined with \eqref{es-2.2-22} gives that
\begin{equation*}
    \| B(t, \cdot)\|_{L^p}^2
 \leq   \|  B_0\|_{L^p}^2\,\exp\left(C\int_0^t \left(1+ \|  u(\tau, \cdot)\|_{\dot{H}^{\gamma_c}}\log(e+\|B(\tau, \cdot)\|_{L^p}^2)\right)\,d\tau\right).
\end{equation*}
 Define $X(t):= \log(e+\|B(t, \cdot)\|_{L^p}^2),$  the above inequality   recasts to
\begin{align*}
    X(t)\leq X(0)+ C\int_0^t \left(1+ \|  u(\tau, \cdot)\|_{\dot{H}^{\gamma_c}}\, X(\tau)\right)\,d\tau.
\end{align*}
Gronwall's lemma then implies that
\begin{equation}\label{es-2.2-33}
    X(t)\leq \left(X(0)+Ct \right)\,\exp\left(C\int_0^t \|  u(\tau, \cdot)\|_{\dot{H}^{\gamma_c}}\,d\tau\right).
\end{equation}
In virtue  of \eqref{en-eq-gamma_c}, we obtain from \eqref{es-2.2-33} that
\begin{align*}
    X(t)\leq   (X(0)+Ct) \exp(C\sqrt{t}\|B_0\|_{L^2})\},
\end{align*}
which gives \eqref{es-2.2-111}.
Finally,   use Young's inequality,  \eqref{es-2.2-33}  and \eqref{en-eq-gamma_c} again,   we infer from  \eqref{es-2.2-22}  that \eqref{es-2.2-000} is satified.
It completes the proof of Proposition \ref{Prop-2.2}.
 \end{proof}

\medbreak
\begin{proof}[\bf Proof of Proposition \ref{Th1}]
Estimate \eqref{blowup-criterion} shows that the local-in-time $H^s$ solution obtained in Proposition \ref{Prop-criteria} may be uniquely continued past to time $T$, if  $\|(\nabla u, \Lambda u)(t, \cdot)\|_{L^\infty}$ is integrable on $[0, T].$  Moreover,  the a priori estimates in
Proposition \ref{Prop-2.2}   give that $\|(\nabla u, \Lambda u)\|_{L^1(0, T; L^\infty)}$ is bounded only in terms of $T$ and norm $\|B_0\|_{L^p}.$ Thus the
global existence of $H^s$ solutions is obtained by contradiction argument.  Taking into account the bounds \eqref{blowup-criterion} and \eqref{es-2.2-000},  we  obtain the bound \eqref{Th1-bound}. The proof of energy equality \eqref{G-en-eq} is classic, we omit its details here.
 \end{proof}
 
\bigbreak
\section{Proof of existence  in Theorem \ref{Th2-1} and \ref{Th2}}\label{S3}

This  section mainly concerns  the proof of existence  of weak solutions of MRE \eqref{MRE} for $\gamma\geq \gamma_c$. Because the initial data is  rough,  the energy inequality \eqref{en-eq} or estimate \eqref{es-2.2-111} do not provide robust compactness properties of the magnetic field by classic compactness argument.  To show strong convergence property of the magnetic field, we develop a compactness result for the corresponding linearized equation, which extends Lions' well-known renormalization argument \cite{PLL} from scalar  transport equation to the active vector equation.

\begin{lemma}\label{Lemma3.1}
Let $d=2, 3$ and $T>0.$ Given a pair of sequence $\{(A_n, v_n)\}_{n\in\N}$   satisfying $$A_n\in\mathcal{C}([0, T]; L^2(\T^d)), \, v_n\in L^2(0, T; H^1(\T^d))\cap L^1(0, T; W^{1, \infty}(\T^d)),$$
and 
\begin{equation}\label{Le3.1-U}
      \|A_n\|_{L^\infty(0, T; L^2(\T^d))}+\|v_n\|_{L^2(0, T; H^1(\T^d))}+\|v_n\|_{L^1(0, T; W^{1, \infty}(\T^d))}\leq M,
    \end{equation}
for some constant $M$   independent of $n$.  Assume that 
\begin{equation*}
\left\{
\begin{aligned}
&\partial_t A_n+\div(A_n \otimes v_n)=\div(v_n\otimes A_n) \quad{\rm{in}}\,\,\mathcal{D}{'}((0, T)\times\T^d),\\
&\div A_n=\div{v}_n=0\hspace*{3cm}~~{\rm{in}}\,\,\mathcal{D}{'}((0, T)\times\T^d),\\
&A_n|_{t=0}=A_{0,n},
\end{aligned}
\right.
\end{equation*}
and
\begin{align*}
 &A_{0,n}\to A_0\quad {\rm{strongly~in}}\quad L^2(\T^d),\\
 &{v}_n\rightarrow {v}\quad\quad\,\,\,{\rm{strongly~in}}\quad L^1(0, T; W^{1, \infty}(\T^d)).
\end{align*}
Then, up to extraction,  $A_n$ converges strongly in $\mathcal{C}([0, T]; L^2(\T^d))$  to a function  $A\in\mathcal{C}([0, T]; L^2(\T^d)),$ which is the unique  solution of the problem
\begin{equation}\label{VE-L}
\left\{
\begin{aligned}
&\partial_t A+\div(A\otimes v)=\div(v\otimes A) \quad{\rm{in}}~~\mathcal{D}{'}((0, T)\times\T^d),\\
&\div A=\div{v}=0\hspace*{2.6cm}{\rm{in}}\,\,\mathcal{D}{'}((0, T)\times\T^d),\\
&A|_{t=0}=A_0 ~~\hspace*{3.5cm}{\rm{a.e.~~ in}}~~\T^d.
\end{aligned}
\right.
\end{equation}
\end{lemma}

The proof of Lemma \ref{Lemma3.1} is postponed to the subsection \ref{pf-Le3.1}.

\subsection{Proof of existence}
In order to prove  the existence parts of Theorem \ref{Th2-1} and Theorem \ref{Th2}, we shall  resort to  the following procedure:
\begin{enumerate}
\item   smooth out the initial data and get a sequence  $\{(B_n, u_n)\}_{n\in\N}$
 of global-in-time smooth solutions to MRE \eqref{MRE};

\item  use compactness results to prove that  $\{(B_n, u_n)\}_{n\in\N}$ converges, up to extraction, 
to a solution of  MRE \eqref{MRE} supplemented with initial data $B_0$;

\item prove  energy inequality.
\end{enumerate}
\medbreak

 For expository purposes, we focus on the case $\gamma=\gamma_c,$ i.e. Theorem \ref{Th2}, just explaining at the end of the subsection what has to be modified to handle the case $\gamma>\gamma_c$. To proceed, let us  smooth out the initial data $B_0\in L^{2+\eps}$ such that
 \begin{equation}\label{cv-data}
    B_{0, n}\in C^\infty(\T^d),\quad B_{0,n}\rightarrow B_0 \quad{\rm strongly~in}\quad L^{2+\eps}(\T^d).
 \end{equation}
Then, in light of the global regularity result stated in Proposition \ref{Th1} there exists a unique global smooth solution $(B_n, u_n)$  satisfying 
\begin{equation}\label{MRE-n}
\left\{\begin{aligned}
 \partial_t B_n+\div(B_n\otimes u_n)  &=\div (u_n \otimes B_n),\\
 (-\Delta)^{\gamma_c} u_n&=\div(B_n\otimes B_n)+\nabla P_n,\\
 \div u_n=\div B_n&=0,
\end{aligned}\right.
\end{equation}
corresponding to initial data $B_{0, n}.$   Being smooth, the pair $(B_n, u_n)$ satisfies all the a priori estimates in Proposition \ref{Prop-gc}, and thus in particular  for all $0<T<\infty,$
\begin{equation}\label{es-2.2-111-n}
    \sup_{t\in{[0, T]}} \| B_n(t, \cdot)\|_{L^{2+\eps}}^2
    \leq \exp\{C(1+T+\| B_{0}\|_{L^{2+\eps}}^2) \exp(C\sqrt{T}\|B_0\|_{L^{2}})\},
\end{equation}
  \begin{multline}\label{es-2.2-000-n}
    \int_0^T  \|(\nabla u_n, \Lambda u_n)(\tau, \cdot)\|_{L^\infty}\,d\tau\\
    \leq C \left(
(1+T+\| B_{0}\|_{L^{2+\eps}}^2)^3\exp(C\sqrt{T}\|B_{0}\|_{L^{2}}) \right).
\end{multline}
And also,
 \begin{align}\label{en-eq-n}
     \|B_n\|_{L^\infty(\R_+; L^2)}^2+  \|u_n\|_{L^2(\R_+; \dot{H}^{\gamma_c})}^2\leq \|B_{0}\|_{L^{2}}^2.
    \end{align}

In the next, we are going to derive strong convergence property of $\{u_n\}_{n\in\N}.$ To do that, we first need to derive the evolution equation of the velocity.  
Noticing  that  the matrix $B_n\otimes B_n$ satisfies 
\begin{multline}\label{eq-dtbb}
 \partial_t(B_n\otimes B_n) +  \div(u_n \odot(B_n\otimes B_n))\\
 =\nabla u_n(B_n\otimes B_n)+ (B_n\otimes B_n) (\nabla u_n)^{\rm T},
\end{multline}
where $\{\div (v \odot(w\otimes z))\}^{i, j}:= \sum_{k=1}^{d} \partial_k    (v^k w^i z^j)$ for vectors $v, w, z\in\R^d.$

\noindent Therefore,
\begin{align*}
    \partial_t u_n &= (-\Delta)^{-\gamma}\mathbb{P}\div  \partial_t(B_n\otimes B_n) \notag \\
    &=  (-\Delta)^{-\gamma}\mathbb{P}\div ( \nabla u_n(B_n\otimes B_n) +(B_n\otimes B_n) (\nabla u_n)^{\rm T}-  \div(u\odot (B_n\otimes B_n))  ).
\end{align*}
Since  $ \gamma=\gamma_c=d/2+1\geq 2,$ one can deduce from the above equality  and the Besov embeeding $L^1(\T^d)\hookrightarrow \dot{B}^{-2\gamma+1}_{2,  1}$ from Proposition \ref{P_Besov}  that
\begin{align}
    \|\partial_t u_n\|_{L^2}&\lesssim \|\nabla u_n(B_n\otimes B_n)\|_{L^1} +\|(B_n\otimes B_n) (\nabla u_n)^{\rm T}\|_{L^1}+ \|u_n\odot (B_n\otimes B_n)\|_{L^1}\notag\\
    &\lesssim \|\nabla u_n\|_{L^\infty}\|B_n\|_{L^2}^2.
\end{align}
In particular, $\{\partial_t u_n\}_{n\in\N}$ is uniformly bounded in $L^1(0, T; L^2),$ thanks to estimates \eqref{es-2.2-000-n}, \eqref{en-eq-n}.   At this stage,  classical compactness arguments imply that there exists a subsequence,  still denoted by $ \{  u_n\}_{n\in\N},$  and a function $$u\in L^2(0, T; H^{\gamma_c})\cap L^1(0, T; W^{1, \infty}),$$
such that 
\begin{align}
     u_n \rightarrow  u \quad{\rm strongly~in}\quad L^2(0, T; L^2).\label{cv1}
\end{align}
Moreover, since $2\gamma_c-1-{2d}/{(2+\eps)}>1,$ then inequalities  \eqref{es-2.2-1}, \eqref{es-2.2-111-n} and  interpolation inequality in  Proposition \ref{P_Besov} together with convergence property \eqref{cv1}  give that
\begin{align}
 u_n \rightarrow   u \quad{\rm strongly~in}\quad L^1(0, T; W^{1, \infty}).\label{cv2}
 \end{align}
 
In virtue of 
convergence properties \eqref{cv2}, \eqref{cv-data} and  estimates \eqref{es-2.2-000-n}, \eqref{en-eq-n}, Lemma \ref{Lemma3.1} implies that, 
up to extraction,  $B_n$ converges strongly in $\mathcal{C}_{\rm loc}([0, \infty); L^2)$  to a function  $$B\in\mathcal{C}_{\rm loc}([0, \infty); L^2)\cap L^\infty_{\rm loc}(0, \infty; L^{2+\eps})\cap L^\infty(0, \infty; L^2),$$ such that the pair $(B, u)$    solves the problem  \eqref{MRE}-\eqref{eq-indata} on $(0,  T)\times \T^d.$ The bound \eqref{es-Th2} is essentially due to estimate \eqref{es-2.2-111-n}.   

In order to prove energy equality \eqref{G-en-eq}, notice that the smooth approximate solution $(B_n, u_n)$ fulfills
\begin{align}\label{en-eq-n00}
     \|B_n(t, \cdot)\|_{L^2}^2+  2\int_0^t\|u_n(\tau, \cdot)\|_{ \dot{H}^{\gamma_c}}^2\,d\tau= \|B_{0, n}\|_{L^{2}}^2.
    \end{align}
Thanks to the strong convergence  property of  the magnetic field $\{B_n\}_{n\in\N}$ in the space $\mathcal{C}_{\rm loc}([0, \infty); L^2(\T^d))$, we have for all $0\leq t<\infty$
\begin{align*}
    \lim_{n\to\infty} \|B_n(t, \cdot)\|_{L^2}^2=\|B(t, \cdot)\|_{L^2}^2.
\end{align*}
As for the velocity field, using the Besov embedding and \eqref{es-2.2-1} we  see that  for any $p\in(2, 4]$
\begin{align*}
    \|u_n\|_{\dot{B}^{\gamma_c+d- {2d}/{p}}_{2, \infty}}\lesssim \|u_n\|_{\dot{B}^{2\gamma_c-1}_{p/2, \infty}}  
  \lesssim  \|B_n\|_{L^{p}}^2.
\end{align*}
This, together with interpolation inequality in  Proposition \ref{P_Besov}  and  convergence property \eqref{cv1} and uniform bound \eqref{es-2.2-111-n} then give that $u_n\rightarrow u$ strongly in $L^2(0, T; H^{\gamma_c}).$  By passing to the limit in \eqref{en-eq-n00}, one can finally conclude that energy equality \eqref{G-en-eq} is true.

\medbreak

Let us now briefly explain how to adapt the proof to the case when $B_0\in L^2$ and $\gamma>\gamma_c.$  First, we smooth out the data  and use the global well-posedness result in \cite[Theorem 3.1]{BFV}, which  gives the existence  of global smooth solutions that  satisfying the energy equality
\begin{align}\label{en-eq-n'}
    \|B_n(t, \cdot)\|_{L^2}^2+  2\int_0^t\|u_n(\tau, \cdot)\|_{ \dot{H}^{\gamma}}^2\,d\tau= \|B_{0, n}\|_{L^{2}}^2.
    \end{align}
Thanks to \eqref{en-eq-n'},  one can get $u_n\rightarrow u$ strongly in $L^2(0, T; H^{\gamma})$  similarly as previous case.  Indeed,   due to $\gamma>\gamma_c$,  this can be seen from  \eqref{es-2.2-1} (taking $p=2$) and interpolation. 
Moreover,  we have the Sobolev embedding  $H^\gamma\hookrightarrow W^{1, \infty}$.
 From this point, the strong convergence in $L^2$ of the magnetic field is obtained  via applying Lemma \ref{Lemma3.1} again.   
Energy equality then follows  by passing to the limit in \eqref{en-eq-n'}.
\medbreak

We complete the proof of existence parts in Theorem \ref{Th2-1} and Theorem \ref{Th2}.

\begin{remark}[regularity property of weak solutions]\label{re-solu}
Here, we list some additional regularity properties of  
 the weak solutions that constructed in this subsection.   
In the case that $\gamma>\gamma_c,$ by testing the equation \eqref{MRE}$_2$ with $u$  we get
\begin{align*}
    \|u\|_{\dot{H}^\gamma}^2&=|\langle B\otimes B, \nabla u\rangle|\\
    &\leq \|B\|_{L^2}^2\|\nabla u\|_{L^\infty}\\
      &\leq \|B\|_{L^2}^2\|u\|_{H^\gamma}.
\end{align*}
Then by Poincar\'{e}'s inequality we have $u\in L^\infty(\R_+; {H}^\gamma).$
  Meanwhile,  we know from estimate \eqref{es-2.2-1} that  $u\in L^\infty(\R_+; \dot{B}^{2\gamma-1}_{1, \infty})$ with the associated mean-free pressure 
      $P\in L^\infty(\R_+; \dot{B}^{0}_{1, \infty}).$
  
  When $\gamma=\gamma_c,$ thanks to  \eqref{es-Th2} and   \eqref{es-2.2-1} we see that
  $$u\in L^\infty(0, \infty; \dot{B}^{2\gamma_c-1}_{(2+\eps)/2, \infty}) \cap  L^\infty(\R_+; \dot{B}^{2\gamma_c-1}_{1, \infty})$$ with the associated mean-free  pressure 
  $$P\in L^\infty( 0, \infty; \dot{B}^{0}_{(2+\eps)/2, \infty})\cap L^\infty(\R_+; \dot{B}^{0}_{1, \infty}).$$
\end{remark}
\bigbreak

\subsection{Proof of Lemma \ref{Lemma3.1}}\label{pf-Le3.1} Finally, we provide the proof of Lemma \ref{Lemma3.1}. Before proceeding to the proof,  let us recall a classical lemma of commutator estimates  for mollifiers.   Let   function $\eta(x)$ be the standard   mollifier on $\T^d,$ for any   function $g\in L^1(\T^d)$ we denote $(g)_\eps: =g*\eta_\eps$ with $\eta_\eps(x):=\eta({x}/{\eps})/{\eps^d}.$   
\begin{lemma}{\cite[Lemma 2.3]{PLL}}\label{Le-regu}
    Let $w\in W^{1, \infty}(\T^d; \R^d),\, g\in L^2(\T^d; \R).$  Then, $(\div(wg))_{\eps}-\div(w\,(g)_\eps)$ converges to 0 in $L^2(\T^d)$ as $\eps$ goes to 0.
\end{lemma}
\bigbreak
\begin{proof}[Proof of Lemma \ref{Lemma3.1}]
The proof is divided into several steps.
\subsubsection*{  Step 1.  Existence}
    At first, from the bound \eqref{Le3.1-U}  and the equality of $\partial_t A_n,$ we have
    \begin{align*}
        \|\partial_t A_n\|_{L^2(0, T;\dot{W}^{-1, 3/2}(\T^d))}&\lesssim \|A_n\otimes v_n\|_{L^2(0, T; L^{{3}/{2}}(\T^d))}+\|v_n\otimes A_n\|_{L^2(0, T; L^{3/2}(\T^d))}\\
        &\lesssim \|A_n\|_{L^\infty(0, T; L^2(\T^d))}\|v_n\|_{L^2(0, T; L^6(\T^d))}\\
        &\lesssim \|A_n\|_{L^\infty(0, T; L^2(\T^d))}\|  v_n\|_{L^2(0, T; H^1(\T^d))}\\
        &\lesssim M^2.
    \end{align*}
    At this stage, classical functional analysis arguments (including the well-known Aubin-Lions-Simon theorem \cite[Theorem II. 5.16]{BF13}) imply that there exists a  subsequence, still denoted by $\{(A_n, v_n)\}_{n\in\N},$ and $$  A \in C_{\rm w}(0, T; L^2(\T^d))\cap\mathcal{C}([0, T]; \dot{H}^{-1}(\T^d))$$ such that
     \begin{align*}
&A_{n}\rightharpoonup    A\quad {\rm{weak}}~*~{\rm{in}}\quad L^\infty(0, T; L^2(\T^d)),\\
&A_{n}\rightarrow   A\quad {\rm{strongly~in}}\quad L^\infty(0, T;  \dot{H}^{-1}(\T^d)),\\
&{v}_n\rightharpoonup {v}~\quad {\rm{weakly}}~{\rm{in}}\quad L^2(0, T; H^1(\T^d)).
\end{align*}
These convergence properties   are enough to enable us     pass to the limit in the  weak formulation of \eqref{VE-L}$_1$: 
\begin{multline*} 
    \int_{\T^d} A_n(t, x)\cdot w(t, x)\,dx-\int_0^t \int_{\T^d}A_n\cdot\partial_t w \,dxd\tau+\int_0^t \int_{\T^d} (v_n \otimes A_n ) : \nabla w \,dxd\tau\\
   =\int_0^t \int_{\T^d} (A_n \otimes v_n ):\nabla w \,dxd\tau+ \int_{\T^d}A_{0,n}(x)\cdot w(0, x)\,dx,
\end{multline*}
for any vector function $w\in \mathcal{C}^\infty([0, T]\times \T^d),$ 
and conclude that $A$ is a solution of problem \eqref{VE-L}.

\subsubsection*{Step 2. Renormalization}  

In this step, we show that
\begin{equation}\label{eq-Le3.1}
    \partial_t |A|^2+ v\cdot\nabla |A|^2=2(A\cdot\nabla v) \cdot A\quad{\rm{in}}~~\mathcal{D}{'}((0, T)\times\T^d).
\end{equation}
In fact, since 
\begin{equation*}
    \partial_t A+\div(A\otimes v)=\div(v\otimes A) \quad{\rm{in}}~~\mathcal{D}{'}((0, T)\times\T^d),
\end{equation*}
testing the above equation by $(A)_\eps\,\phi$ with $\phi$ any test function on $[0, T]\times\T^d,$ we have 
\begin{align} \label{Lemma 3.1-000}
    &\int_{\T^d} A(t, x)\cdot(A)_\eps(t, x)\phi(t, x)\,dx-\int_{\T^d} A_0(x)\cdot(A_0)_\eps( x)\phi(0, x)\,dx\notag\\
    &\qquad\quad-\int_0^t \int_{\T^d}\left(A \cdot(A)_\eps\,\partial_t \phi +
   (A\cdot\nabla v)_{\eps} \cdot A\,\phi +(A\cdot\nabla v)\cdot(A)_\eps \phi \right)\,dxd\tau\notag\\
   =&  -\int_0^t \int_{\T^d} \left((v\cdot\nabla A)_\eps\cdot A\,\phi +(v\cdot\nabla A)\cdot (A)_\eps\,\phi\right)\,dxd\tau.
\end{align}
In above, we used  equality $\partial_t (A)_\eps=-(v\cdot\nabla A)_\eps +(A\cdot\nabla v)_\eps.$

Now, using that $\div v=0$  one has
\begin{align*}
     &\int_0^t \int_{\T^d} \left((v\cdot\nabla A)_\eps\cdot A\,\phi -(v\cdot\nabla (A)_\eps))\cdot A\,\phi\right)\,dxd\tau\\
    =&\int_0^t \int_{\T^d} \left(\div(A\otimes v)_\eps-\div((A)_\eps\otimes v) \right) \cdot A \,\phi \,dxd\tau\\
    =&\sum_{i=1}^d \int_0^t \int_{\T^d} \left(\div(A^i v)_\eps-\div((A^i)_\eps  v) \right) \, A^i\, \phi \,dxd\tau.
\end{align*}
  H\"older's inequality   and Lemma \ref{Le-regu} with the  assumptions that $v\in L^1(0, T; W^{1, \infty}(\T^d)), \, A\in L^\infty(0, T; L^2(\T^d))$ then  yield that
 \begin{align*}
     &\left|\int_0^t \int_{\T^d} \left((v\cdot\nabla A)_\eps\cdot A\,\phi +(v\cdot\nabla A)\cdot (A)_\eps\,\phi\right)\,dxd\tau\right| \\
     \leq&   \sum_{i=1}^d\left| \int_0^t \int_{\T^d} \left((\div(A^i v))_\eps-\div((A^i)_\eps  v) \right) \, A^i\, \phi \,dxd\tau\right|\\
     \lesssim&  \sum_{i=1}^d \|(\div(A^i v))_\eps-\div((A^i)_\eps  v)\|_{L^1(0, T; L^2(\T^d))} \|A^i\|_{L^\infty(0, T; L^2(\T^d))} \|\phi\|_{L^\infty([0, T]\times \T^d)}\\
     &\quad\rightarrow 0,\quad {\rm as}~~\eps\to0.
\end{align*}
Thus   for the  last line in equality \eqref{Lemma 3.1-000},
\begin{align*}
    &-\int_0^t \int_{\T^d} \left((v\cdot\nabla A)_\eps\cdot A\,\phi +(v\cdot\nabla A)\cdot (A)_\eps\,\phi\right)\,dxd\tau\\
    =& -\int_0^t \int_{\T^d} \left((v\cdot\nabla A)_\eps\cdot A\,\phi -(v\cdot\nabla (A)_\eps))\cdot A\,\phi\right)\,dxd\tau
+ \int_0^t \int_{\T^d} (v\cdot\nabla \phi) A\cdot (A)_\eps\,dx d\tau\\
&\quad\rightarrow \int_0^t \int_{\T^d} (v\cdot\nabla \phi) |A|^2\,dx d\tau ,\quad {\rm as}~~\eps\to0,
\end{align*}
while the remaining terms in equality \eqref{Lemma 3.1-000} converge  obviously. Precisely, one  has
\begin{multline}\label{wf-A2}
    \int_{\T^d} |A(t, x)|^2 \phi(t, x)\,dx-\int_{\T^d} |A_0(x) |^2\phi(0, x)\,dx\\
    =\int_0^t \int_{\T^d}\left(|A|^2\partial_t \phi +
   2(A\cdot\nabla v)  \cdot A\,\phi+(v\cdot\nabla \phi) |A|^2 \right)\,dxd\tau,
\end{multline}
i.e.  \eqref{eq-Le3.1} is satisfied.

\subsubsection*{Step 3. Continuity in time and uniqueness} 
To show that $A\in \mathcal{C}([0, T]; L^2),$  we first prove $\lim_{t\to 0+}\|A(t, \cdot)-A_0(\cdot)\|_{L^2(\T^d)}=0.$ In fact, by taking $\phi\equiv 1$ in \eqref{wf-A2}, we  see that 
\begin{align*}
     &\left|\int_{\T^d} |A(t, x)|^2  \,dx-\int_{\T^d} |A_0(x) |^2 \,dx\right| \notag\\
    =&2\left|\int_0^t \int_{\T^d} 
   (A\cdot\nabla v)  \cdot A\,dxd\tau\right|\\
   \lesssim& \|A\|_{L^\infty(0, T; L^2(\T^d))}^2\|\nabla v\|_{L^1(0, t; L^\infty(\T^d))} \quad\to 0,\quad {\rm as}~t\to0+.
\end{align*}
With above estimate and the fact that $A\in\mathcal{C}_{\rm w}([0, T]; L^2(\T^d)),$ we get strong right continuity  of $A$ at $t=0.$  A standard connectivity  argument enables us to conclude that $A\in \mathcal{C}([0, T]; L^2(\T^d)).$

Next, we explain why $ A$ is the unique solution   of problem \eqref{VE-L}  in the class $L^\infty(0, T; L^2(\T^d)).$  By contradiction, assume that there exists another solution $\tilde A \in L^\infty(0, T; L^2(\T^d))$ with same initial data $A_0.$ Then 
\begin{equation*}
    \partial_t (A-\tilde{A})+\div((A-\tilde{A}))\otimes v)=\div(v\otimes (A-\tilde{A})) \quad{\rm{in}}~~\mathcal{D}{'}((0, T)\times\T^d).
\end{equation*}
Similar to  \eqref{eq-Le3.1}, this implies that 
\begin{equation*}
    \partial_t |A-\tilde{A}|^2+  v\cdot\nabla|A-\tilde{A}|^2 = 2((A-\tilde{A})\cdot\nabla v)\cdot (A-\tilde{A})  \quad{\rm{in}}~~\mathcal{D}{'}((0, T)\times\T^d),
\end{equation*}
and integrating in space-time further gives that 
\begin{align*}
     \int_{\T^d} |A(t, x)-\tilde{A}(t, x)|^2  \,dx
    =&2\int_0^t \int_{\T^d}
    ((A-\tilde{A})\cdot\nabla v)  \cdot (A-\tilde{A}) \,dx d\tau\\
   \lesssim& \int_0^t \|(A-\tilde{A})(\tau, \cdot)\|_{L^2(\T^d)}^2\|\nabla v(\tau, \cdot)\|_{L^\infty(\T^d)}\,d\tau.
\end{align*}
By Gronwall's lemma, one gets $A\equiv \tilde{A}$ on $[0, T]\times \T^d.$

\subsubsection*{Step 4. Strong convergence} Let us use the notation $\bar f$ to denote the weak limit of $\{f_n\}_{n\in\N}$ in $L^1(\T^d).$
Due to the fact that convex lower semicontinuous functions give rise to $L^1$-sequentially weakly lower semicontinuous functionals (see e.g. \cite[Theorem 3]{NP18}), we have 
\begin{align}\label{lsc-1}
    |f|^2\leq \overline{|f|^2} \quad{\rm a.e. ~on}~~ \T^d,
\end{align}
for any sequence $\{f_n\}_{n\in\N}$ such that
\begin{align*}
    f_n\rightharpoonup  f \quad{\rm and}\quad |f_n|^2\rightharpoonup \overline{|f|^2}\quad{\rm weakly~~in}\quad L^1(\T^d).
\end{align*}

In what follows, we show 
\begin{align}\label{3.1-111}
    |\overline{A\otimes A}-A\otimes A|\leq C(\overline{|A|^2}-|A|^2)\quad{\rm a.e. ~on}~~[0, T]\times\T^d.
\end{align}
For fixed $i, j\in\{1, \cdots,d\},$ using  identity
\begin{multline*}
   2 (A^i_n A^j_n- A^i A^j)= (A^i_n+A^j_n)^2-(A^i+A^j)^2\\-\left((A^i_n)^2-(A^i)^2+(A^j_n)^2-(A^j)^2\right),
\end{multline*}
we see that
\begin{multline*}
  2(\overline{A^iA^j}- A^i A^j)
   =  \overline{(A^i+A^j)^2}- (A^i+A^j)^2\\-\left(\overline{(A^i)^2}- (A^i)^2+\overline{(A^j)^2}-(A^j)^2\right).
\end{multline*}
Thanks to \eqref{lsc-1}, we get 
\begin{align}
    2(\overline{A^iA^j}- A^i A^j)
     \geq& -\left(\overline{(A^i)^2}- (A^i)^2+\overline{(A^j)^2}-(A^j)^2\right)\notag\\
     \geq&-2\overline{(|A|^2}- |A|^2).\label{3.1-lower}
\end{align}
Similarly,  using  identity
\begin{multline*}
   2 (A^i_n A^j_n- A^i A^j)= -(A^i_n-A^j_n)^2+(A^i-A^j)^2\\
   +(A^i_n)^2-(A^i)^2+(A^j_n)^2-(A^j)^2,
\end{multline*}
we see that
\begin{multline*}
  2(\overline{A^iA^j}- A^i A^j)
   =  -\overline{(A^i-A^j)^2}+ (A^i-A^j)^2\\+ \overline{(A^i)^2}- (A^i)^2+\overline{(A^j)^2}-(A^j)^2.
\end{multline*}
Thanks to \eqref{lsc-1} again, we get
\begin{align}
    2(\overline{A^iA^j}- A^i A^j)\leq&  ~\overline{(A^i)^2}- (A^i)^2+ \overline{(A^j)^2}  -(A^j)^2\notag\\
    \leq&~ 2(\overline{|A|^2}- |A|^2).\label{3.1-up}
\end{align}
Combining \eqref{3.1-lower} and \eqref{3.1-up} gives rise to \eqref{3.1-111}.
 
Now, testing the equation \eqref{VE-L}$_1$ by $A_n$ and using that $\div v_n=\div A_n=0,$ one gets
 \begin{align*}
     \int_{\T^d} |A_n(t, x)|^2  \,dx=\int_{\T^d} |A_{0, n}(x) |^2 \,dx-\int_0^t \int_{\T^d} (A_n\otimes A_n):\nabla v_n\,dxd\tau.
 \end{align*}
 Recall from \eqref{wf-A2} that 
 \begin{align*}
     \int_{\T^d} |A(t, x)|^2  \,dx=\int_{\T^d} |A_0(x) |^2 \,dx-\int_0^t \int_{\T^d} (A\otimes A):\nabla v\,dxd\tau.
 \end{align*}
 From above two equalities, one finds that
 \begin{multline}\label{eq-deltaA}
 \|A_n(t, \cdot)\|_{L^2(\T^d)}^2-\|A(t, \cdot)\|_{L^2(\T^d)}^2=\|A_{0, n} \|_{L^2(\T^d)}^2 -\|A_0\|_{L^2(\T^d)}^2\\
 +\int_0^t \int_{\T^d} \left((A\otimes A):\nabla (v-v_n)+ (A\otimes A-A_n\otimes A_n):\nabla v_n\right)\,dxd\tau.
\end{multline}
Thanks to  our assumptions in Lemma \ref{Lemma3.1}, we know that the first and second term on the right-hand side of \eqref{eq-deltaA}  converges to 0 as $n\to \infty,$ respectively,   while the third term converges to
\begin{align*}
    \int_0^t \int_{\T^d} (A\otimes A-\overline{A\otimes A}):\nabla v\,dxd\tau.
\end{align*}
Thus pass to the limit in \eqref{eq-deltaA} yields
\begin{align*}
   \int_{\T^d}  \overline{|A(t, x)|^2}  \,dx&=  \lim_{n\to\infty} \|A_n(t, \cdot)\|_{L^2(\T^d)}^2 \notag\\
   &=\|A(t, \cdot)\|_{L^2(\T^d)}^2 + \int_0^t \int_{\T^d} (A\otimes A-\overline{A\otimes A}):\nabla v\,dxd\tau.
\end{align*}
 Inequality \eqref{3.1-111} and H\"older's inequality then imply that
 \begin{multline*}
 \int_{\T^d} \left( \overline{|A(t, x)|^2}-|A(t, x)|^2  \right)\,dx\\
 \lesssim  
\int_0^t\left( \int_{\T^d} \left( \overline{|A(\tau, x)|^2}-|A(\tau, x)|^2  \right)\,dx\right)\|\nabla v(\tau, \cdot)\|_{L^\infty}\,d\tau,
 \end{multline*}
from which Gronwall's lemma finally yields
\begin{align*}
    \lim_{n\to\infty} \|A_n(t, \cdot)\|_{L^2(\T^d)}^2 =\|A(t, \cdot)\|_{L^2(\T^d)}^2\quad {\rm a.e.~on}~~ [0, T].
\end{align*}
This  convergence property combined with the bound \eqref{Le3.1-U} and previously established $A\in\mathcal{C}([0, T]; L^2)$
yield the strong convergence  $A_n\rightarrow A$ in $\mathcal{C}([0, T]; L^2).$ We  complete the proof of the lemma.
\end{proof}

\section{Proof of  uniqueness}\label{S4}
This  section concerns  the proof of uniqueness in Theorem \ref{Th2-1} and \ref{Th2}.
Because the magnetic field is rough,  we shall first prove uniqueness of weak solutions written in the Lagrangian coordinates and then obtain uniqueness in the Eulerian coordinates by the equivalence of the two formulations in our framework.   To this end,   we first present  here the  definition and properties of the Lagrangian change of variables.   More details and proofs can be found in e.g. \cite{DM1, DM2}.

\subsection{The Lagrangian coordinates}
Let $X $ be the flow associated to the divergence free velocity field $v,$ that is, the solution to
 \begin{equation}\label{def-trajectory}
     \frac{d}{d t}X_{v} (t,  y)=   {v}(t, X_{v}(t,  {y})), \quad X_{v} (0,  y)= {y}, \quad\forall\,  {y}\in\T^d.
 \end{equation}
Equation \eqref{def-trajectory} describes the relation between the Eulerian coordinates ${x}= X_{v}(t, {y})$ and the Lagrangian coordinates ${y}.$ In the sequel, we shall use the notation $\underline{v}(t,  {y}):=v(t, X_{v}(t,  {y})).$  We also define by $\Delta_{\underline{v}}$, $\nabla_{\underline{v}}$ and $\mathrm{div}_{\underline{v}}$ the operators corresponding to the original operators $\Delta$, $\nabla$ and ${\rm div}$ after performing the change to the Lagrangian coordinates,  respectively.   Index $\underline{{v}}$ underlines the dependency on $\underline{{v}}$.  The chain rule yields that  
\begin{align}\label{Lag-formula1}
       \nabla_{\underline{v}} a=\mathrm{M}_{\underline{v}}^{\rm{T}} \nabla_{y}  a ={\rm div}_{y} ( \mathrm{M}_{\underline{v}}^{\rm T}\, a)\quad {\rm  for ~any~scalar ~function}~a,
\end{align}
and
\begin{equation}\label{Lag-formula}
   {\mathrm{div}}_{\underline{v}} w=\mathrm{M}_{\underline{v}}^{\rm T}:\nabla_{y} w={\rm div}_{y}(\mathrm{M}_{\underline{v}}\, w),\quad    \nabla_{\underline{v}}  w=(\nabla_y w)\mathrm{M}_{\underline{v}}
\end{equation}
 for any (column) vector field $w\in\R^d,$ and 
\begin{equation*} 
   {\mathrm{div}}_{\underline{v}} {\rm A}= {\rm div}_{y}(  {\rm A}\,  \mathrm{M}_{\underline{v}}^{\rm T})  \quad {\rm for ~any~matrix}~\mathrm{A}\in\R^d\times\R^d,
\end{equation*}
where  $\mathrm{M}_{ \underline{v}}(t, y):= (\nabla_{y} X_{v}(t, y))^{-1}.$   
These algebraic relations will be of fundamental importance in our analysis.
\medbreak

Let us now list a few basic properties for the Lagrangian change of variables.
\begin{proposition}\label{Lag-es1}
  Assume that $v\in L^1(0, T; W^{1, \infty}).$  Then the solution to the system \eqref{def-trajectory} exists on the time interval $[0, T)$  and 
  \begin{equation*}
    \nabla_{y}X_v(t, y)={\rm Id}+\int_0^t( \nabla_{x}   {v})(\tau, X_v(\tau,  y)) \, \nabla_{y}X_v(\tau, y)   \,d\tau,
\end{equation*}
  thus
\begin{equation*}
 \|\nabla_{y} X_{v}(t)\|_{L^\infty }\leq \exp\Bigl(\int_0^t\|\nabla_{x} v(\tau, \cdot)\|_{L^\infty}\,d\tau  \Bigr).
\end{equation*}
Let ${Y}_{v}(t, \cdot)$ be the inverse diffeomorphism of $X_{v}(t, \cdot),$ then
\begin{equation*}
   \nabla_{x} {Y}_{v}(t, x) = (\nabla_{y} X_{v}(t, y))^{-1}.
\end{equation*}
Furthermore, if
\begin{equation}\label{La-smallness}
    \int_0^t \|\nabla_{y} \underline{v}(\tau, \cdot)\|_{\dot{B}^{{d}/{q}}_{q, 1}}\,d\tau\leq \frac{1}{2}\quad {\rm for~some}~q\in[1, \infty),
\end{equation}
then $\mathrm{M}_{ \underline{v}}(t, y)$ satisfies 
\begin{equation}\label{formula-M}
\mathrm{M}_{ \underline{v}}(t, y)= \bigl({\rm Id}+(\nabla_{y} X_{v}-{\rm Id})\bigr)^{-1}=\sum_{j=0}^\infty (-1)^j\,\Bigl(\int_0^t \nabla_{y} \underline{v}\,d\tau\Bigr)^j,
\end{equation}
and  the following inequalities:
\begin{align*}
    \|\mathrm{M}_{ \underline{v}}-{\rm Id}\|_{L^\infty(0, t;  \dot{B}^{ {d}/{q}}_{q, 1})} \lesssim & \,  \|\nabla_{y} \underline{v} \|_{L^1(0, t; \dot{B}^{{d}/{q}}_{q, 1})},\\
     \|\nabla_{y}\mathrm{M}_{\underline{v}} \|_{L^\infty(0, t; \dot{B}^{ {d}/{q}}_{q, 1})}\lesssim& \,  \|\nabla^2_{y} \underline{v} \|_{L^1(0, t; \dot{B}^{{d}/{q}}_{q, 1})}.
\end{align*}
\end{proposition}

\bigbreak

  In Lagrangian coordinates, we will use repeatedly the fact that  $\delta \mathrm{M}_{\underline{v}}(t, y):=\mathrm{M}_{\underline{v_1}}(t, y)-\mathrm{M}_{\underline{v_2}}(t, y)$ satisfies
\begin{equation}\label{formula-M2}
    \delta \mathrm{M}_{\underline{v}}(t, y) =
     (\mathrm{D}_1(t, y)-\mathrm{D}_2(t, y))\cdot\left(\sum_{k\geq1}(-1)^k\sum_{j=0}^{k-1}  D_1^j(t, y) D_2^{k-j-1} (t, y)  \right)
\end{equation}
with $\mathrm{D}_i(t, y):=\int_0^t  \nabla_{y} \underline{ v_i}(\tau, y)\,d\tau$ for $i=1, 2.$

Then, we have the proposition below.
\begin{proposition}\label{Lag-es2}
 Let $\underline{v_1}$ and $\underline{v_2}$ be two vector fields satisfying \eqref{La-smallness} (with same $q$) and define $\delta\underline{v}:=\underline{v_1}-\underline{v_2}.$ Then
 \begin{align*}
    \|\delta \mathrm{M}_{\underline{v}}\|_{L^\infty(0, t; \dot{B}^{{d}/{q}}_{q, 1})}\lesssim & \,   \|\nabla_{y} \delta\underline{v}\|_{L^1(0, t; \dot{B}^{{d}/{q}}_{q, 1})},\\
     \|\nabla\delta \mathrm{M}_{\underline{v}}\|_{L^\infty(0, t; \dot{B}^{{d}/{q}}_{q, 1})}\lesssim & \,  \|\nabla^2_{y} \delta\underline{v}\|_{L^1(0, t; \dot{B}^{{d}/{q}}_{q, 1})}\\
     &\quad+\|\nabla_{y} \delta\underline{v}\|_{L^1(0, t; \dot{B}^{{d}/{q}}_{q, 1})} \|(\nabla^2_{y}  \underline{ v_1}, \nabla^2_{y}  \underline{v_2})\|_{L^1(0, t; \dot{B}^{{d}/{q}}_{q, 1})}.
\end{align*}
\end{proposition}

\bigbreak
\bigbreak

Let us now  derive the MRE \eqref{MRE} in the  Lagrangian coordinates.  We set
\begin{align*} 
   &\underline{B}(t, y):= B(t, X_{{u}}(t, {y})),\quad \underline{u}(t, y):=u(t,X_{{u}}(t, {y}))\\
   &\underline{P}(t, y):= P(t, X_{{u}}(t, {y})).
\end{align*}
Using the fact that solutions of the magnetic equation \eqref{MRE}$_1$ is given by 
\begin{align*}
    \underline{B}(t, y)=B_0(y)\cdot \nabla_y X_{u}(t, y),
\end{align*}  
and relations around \eqref{Lag-formula}, we find that MRE \eqref{MRE}-\eqref{eq-indata} recasts to 
\begin{equation}\label{MRE-Lag}
    \left\{\begin{aligned}
       \underline{B}(t, y)&= B_0(y) + \int_0^t B_0(y)\cdot \nabla_y \underline{u}(\tau, y)\,d\tau,\\
        (-\Delta_{y})^{\gamma} \underline{u}-\mathrm{M}^{\rm T}_{\underline u}\nabla_{y}\underline{P}&=\mathrm{div}_{y}\left((\underline{B}\otimes \underline{B})\mathrm{M}_{\underline{u}}^{\rm T}\right)+ ( (-\Delta_{y})^{\gamma}-(-\Delta_{\underline u})^\gamma) \underline{u},\\
        \mathrm{div}_y\underline{u}&= \left({\rm Id}-\mathrm{M}^{\rm T}_{\underline u}\right):\nabla_y \underline{u}.
    \end{aligned} \right.
\end{equation}

 We are ready  to prove the  uniqueness parts of Theorem \ref{Th2-1}  and Theorem \ref{Th2}.  Consider two solutions $(B_1, u_1,   P_1)$ and $(B_2, u_2,   P_2)$  obtained from Theorem \ref{Th2-1} (respectively, Theorem \ref{Th2}) for  MRE \eqref{MRE}, emanating from the same initial data $B_0$, and denote by  $(\underline{B_1}, \underline{u_1},  \underline{P_1})$ and  $(\underline{B_2}, \underline{u_2},  \underline{P_2})$ the corresponding triplets in Lagrangian coordinates.  That is 
  \begin{align*} 
   &\underline{B_i}(t, y):= B_i(t, X_{{u_i}}(t, {y}),\quad \underline{u_i}(t, y):=u_i(t,X_{{u_i}}(t, {y}))\\
   &\underline{P_i}(t, y):= P_i(t, X_{{u_i}}(t, {y}),
\end{align*}
  where $X_{u_i}$ is the solution of \eqref{def-trajectory} with velocity $u_i$, for $i=1, 2.$
Recall from Remark \ref{re-solu} that  we have 
  \begin{align*}
       u_{i}\in L^\infty(0, \infty; \dot{B}^{2\gamma-1}_{p, \infty})\quad{\rm for}~~ i=1, 2,
  \end{align*}
when  $\gamma>\gamma_c$ and $p=1$ (respectively,  when $\gamma=\gamma_c$ and $p=(2+\eps)/2$).  Since in our framework the Lagrangian and Eulerian formulations are equivalent,
hence taking $T_0$ small enough, we may assume, with no loss of generality, that
\begin{equation}\label{Lag-smallsolu}
   C(T_0):= \int_0^{T_0} \|  \underline{ u_1}(t, \cdot)\|_{ \dot{B}^{2\gamma-1}_{p, \infty} }\,dt + \int_0^{T_0} \|  \underline{ u_2}(t, \cdot)\|_{ \dot{B}^{2\gamma-1}_{p, \infty} }\,dt\leq \frac{1}{2}.
\end{equation}
Particularly,  this ensures inequality \eqref{La-smallness} is satisfied and thus Proposition \ref{Lag-es1}  and \ref{Lag-es2} apply. Indeed,   since $2\gamma-1>d/p+1$ is satisfied for both cases,  the interpolation inequality  from Proposition \ref{P_Besov} implies that for any function $w\in\dot{B}^{2\gamma-1}_{p, \infty}$ which has zero mean on $\T^d,$ it satisfies 
\begin{align}\label{U-es000}
      \|w\|_{\dot{B}^{d/{p}+1}_{p, 1}}\lesssim    \|w\|_{\dot{B}^{2\gamma-1}_{p, \infty}}.
\end{align}

Before proceeding,  we have to emphasize that the proof will 
be divided into two parts $\gamma\in\N$ and $\gamma\notin\N$. This seems unavoidable, since in the case $\gamma\notin\N$ the operator $(-\Delta)^{\gamma}$ is {\em{nonlocal}}\, so does  the associated operator $(-\Delta_{\underline u})^{\gamma}.$  This fact gives us    difficulties when performing difference estimates. For expository purposes,  we focus our attention on Theorem \ref{Th2-1},  i.e. $\gamma>\gamma_c,~p=1.$ In the end, we will explain   how to adapt the argument to  prove uniqueness  in  the case  $\gamma=\gamma_c$ and $p=(2+\eps)/2.$

\subsection{The case \texorpdfstring{$\gamma\in\N$}{TEXT}}
Define $(\delta \underline{B},   \delta \underline{u}, \delta \underline{P}):= (\underline{B_1}-\underline{B_2},  \underline{u_1}-\underline{u_2},  \underline{P_1}-\underline{P_2})$ and $\delta{\rm M}:= {\rm M}_{\underline{u_1}}-{\rm M}_{\underline{u_2}}.$ In virtue of formulas \eqref{Lag-formula}-\eqref{Lag-formula1}, the system for $(\delta\underline{B}, \delta\underline{u}, \delta\underline{P})$ reads
\begin{equation}\label{MRE-Lag-d}
    \left\{\begin{aligned}
       \delta \underline{B}(t, y)&=   \int_0^t B_0(y)\cdot \nabla_y \delta\underline{u}(\tau, y)\,d\tau,\\
        (-\Delta_{y})^{\gamma} \delta\underline{u}- \nabla_{y}\delta\underline{P}&=R_1+R_2+R_3,\\
        \mathrm{div}_y\delta\underline{u}&=  R_4,
    \end{aligned} \right.
\end{equation}
where
\begin{align*}
    R_1:=&\, {\rm div}_y(\delta{\rm M}^{\rm T}   \underline{P_1})+ {\rm div}_y(({\rm M}_{\underline{u_2}}^{\rm T}  -{\rm Id}) \delta\underline{P}),\\
    R_2:=&\, {\rm div}_y ((\underline{B_1}\otimes \underline{B_1})\delta {\rm M}^{\rm T})+ \div_y ( (\delta \underline{B}\otimes \underline{B_1}){\rm M}_{\underline{u_2}} ^{\rm T}+(\underline{B_2}\otimes \delta \underline{B}) {\rm M}_{\underline{u_2}}^{\rm T}),\\
    R_3:=& \,((-\Delta_{\underline{u_2}})^\gamma-(-\Delta_{\underline{u_1}})^\gamma)\underline{u_2}+((-\Delta_y)^\gamma-(-\Delta_{\underline{u_1}})^\gamma)\delta \underline{u},\\
    R_4:=&\,  ({\rm Id}- {\rm M}_{\underline{u_2}}^{\rm T}) :\nabla_y  \delta\underline{u}- \delta{\rm M}^{\rm T} :\nabla \underline{u_1}.
\end{align*}

Thanks to H\"{o}lder's inequality,  the embedding $\dot{B}^{d/p}_{p, 1}\hookrightarrow L^\infty$  and \eqref{U-es000},  one has 
\begin{align}
   \|\delta \underline{B}(t, \cdot)\|_{L^2}&\leq \|B_0\|_{L^2}\|\nabla_y \delta\underline{u}\|_{L^1(0, t; L^{\infty})}\notag\\
   &\leq \|B_0\|_{L^2} \| \delta\underline{u}\|_{L^1(0, t; {\dot{B}^{d+1}_{1, 1}})}\notag\\
    &\lesssim \|B_0\|_{L^2}\| \delta\underline{u}\|_{L^1(0, t; {\dot{B}^{2\gamma-1}_{1, \infty}})}\label{U-es1}
    \end{align}
    on $[0, T_0].$
Meanwhile, applying  Lemma \ref{Le-L} to the equation of $\delta \underline{u}$ gives
   \begin{align}\label{U-es2}
        \|\delta\underline{u}(t, \cdot)\|_{\dot{B}^{2\gamma-1}_{1, \infty}} +\| \delta\underline{P}\|_{\dot{B}^{0}_{1, \infty}}
        \lesssim \|(R_1, R_2, R_3)(t, \cdot)\|_{\dot{B}^{-1}_{1, \infty}} +\|R_4(t, \cdot)\|_{\dot{B}^{2\gamma-2}_{1, \infty}}.
   \end{align}
   
In what follows, we   estimate the right-hand side  of \eqref{U-es2} term by term.  For that purpose,  we recall the following product law  from Lemma \ref{Le-PL}:
\begin{align}\label{productlaw1}
    \dot{B}^{d }_{1, 1}\times \dot{B}^{0}_{1, \infty}\hookrightarrow \dot{B}^{0}_{1, 1}.
\end{align}
Then, by \eqref{productlaw1} and Proposition \ref{Lag-es1} and \ref{Lag-es2}, for the term $R_1$   we have
\begin{align}\label{U-esR1}
    \|R_1(t, \cdot)\|_{\dot{B}^{-1}_{1, \infty}}
    \lesssim&~ \|\delta{\rm M}\|_{\dot{B}^{d}_{1, 1}}\|  \underline{P_1}\|_{\dot{B}^{0}_{1, \infty}}+ \|{\rm M}_{\underline{u_2}} -{\rm Id}\|_{\dot{B}^{d}_{1, 1}} \|\delta\underline{P}\|_{\dot{B}^{0}_{1, \infty}}\notag\\
 \lesssim&~ \|  \underline{P_1}\|_{\dot{B}^{0}_{1, \infty}}\|  \delta\underline{u}\|_{L^1(0, t;{\dot{B}^{d+1}_{1, 1}})}+ \| {\underline{u_2}}\|_{L^1(0, t;   {\dot{B}^{d+1}_{1, 1}})} \|\delta\underline{P}\|_{\dot{B}^{0}_{1, \infty}}\notag\\
  \lesssim&~ \|  \underline{P_1}\|_{\dot{B}^{0}_{1, \infty}}\|  \delta\underline{u}\|_{L^1(0, t;{\dot{B}^{2\gamma-1}_{1, \infty}})}+ C(T_0)\|\delta\underline{P}\|_{\dot{B}^{0}_{1, \infty}},
\end{align}
where in the last line inequality   \eqref{U-es000} was used again.
Let us now bound $R_2.$  It is clear  that
\begin{align*}
    \|{\rm div}_y ((\underline{B_1}\otimes \underline{B_1})\delta {\rm M}^{\rm T})\|_{\dot{B}^{-1}_{1, \infty}}&\lesssim \|\delta {\rm M}\|_{\dot{B}^{d}_{1, 1}} \|B_1\|_{\dot{B}^{0}_{1, \infty}}^2\\
    &\lesssim \|B_1\|_{L^2}^2\|  \delta\underline{u}\|_{L^1(0, t;{\dot{B}^{2\gamma-1}_{1, \infty}})}.
    \end{align*}
Then, we write  the decomposition
    \begin{align*}
        {\rm div}_y ((\delta \underline{B}\otimes \underline{B_1}) {\rm M}_{\underline{u_2}}^{\rm T})=  {\rm div}_y ( (\delta \underline{B}\otimes \underline{B_1})( {\rm M}_{\underline{u_2}}^{\rm T}-{\rm Id}) +  \delta \underline{B}\otimes \underline{B_1})
    \end{align*}
    and find that
\begin{align*}
\|\div_y ((\delta \underline{B}\otimes \underline{B_1})({\rm M}_{\underline{u_2}}^{\rm T}-{\rm Id}))\|_{\dot{B}^{-1}_{1, \infty}}&\lesssim
\|{\rm M}_{\underline{u_2}}-{\rm Id}\|_{\dot{B}^{d}_{1, 1}}\|\delta \underline{B} \otimes \underline{B_1}\|_{\dot{B}^{0}_{1, \infty}}\\
&\lesssim
C(T_0)\| \underline{B_1}\|_{L^2}\|\delta \underline{B}\|_{L^2}\\
 \|\div_y (  \delta \underline{B}\otimes \underline{B_1})\|_{\dot{B}^{-1}_{1, \infty}}&  \lesssim  \|\delta \underline{B}\otimes \underline{B_1}\|_{L^1}\\
 &\lesssim \|\underline{B_1}\|_{L^2} \|\delta \underline{B}\|_{L^2}.
\end{align*}
This yields that
\begin{align*}
    \|\div_y ( (\delta \underline{B}\otimes \underline{B_1}){\rm M}_{\underline{u_2}}^{\rm T} )\|_{\dot{B}^{-1}_{1, \infty}}&\lesssim  (1+ C(T_0))\| \underline{B_1}\|_{L^2}\|\delta \underline{B}\|_{L^2},
\end{align*}
and similarly for the last term in $R_2$, that is
  \begin{align*}
    \|\div_y ( ( \underline{B_2}\otimes \delta \underline{B}){\rm M}_{\underline{u_2}}^{\rm T} )\|_{\dot{B}^{-1}_{1, \infty}}&\lesssim  (1+ C(T_0))\| \underline{B_2}\|_{L^2}\|\delta \underline{B}\|_{L^2},
\end{align*}
from which one has
\begin{align}\label{U-esR2}
    \|R_2\|_{\dot{B}^{-1}_{1, \infty}} 
    \lesssim \|B_1\|_{L^2}^2\|  \delta\underline{u}\|_{L^1(0, t;{\dot{B}^{2\gamma-1}_{1, \infty}})}+  (1+ C(T_0))\| (\underline{B_1}, \underline{B_2})\|_{L^2}\|\delta \underline{B}\|_{L^2}.
\end{align}

Next, we estimate $R_4.$ Obviously, we have
\begin{align*}
    \|w\|_{L^\infty} \lesssim \|w\|_{\dot{B}^{d}_{1, 1}} \lesssim  \|w\|_{\dot{B}^{2\gamma-2}_{1, \infty}}.
\end{align*}
Then Lemma \ref{Le-KP2} implies that $\dot{B}^{2\gamma-2}_{1, \infty}$ is an algebra for mean-free functions on $\T^d.$ Thus by formulas \eqref{formula-M}, \eqref{formula-M2} and  inequality \eqref{Lag-smallsolu},  we have
\begin{align}
     \|R_4\|_{\dot{B}^{2\gamma-2}_{ 1, \infty}} 
     \lesssim&~ \| {\rm Id}- {\rm M}_{\underline{u_2}}^{\rm T} :\nabla_y \delta\underline{u}\|_{\dot{B}^{2\gamma-2}_{ 1, \infty}}+ \|\delta{\rm M}^{\rm T} :\nabla_y \underline{u_1}\|_{\dot{B}^{2\gamma-2}_{ 1, \infty}}\notag\\
     \lesssim&~ \|  \underline{u_2}\|_{L^1(0, t;  \dot{B}^{2\gamma-1}_{1, \infty})}\|  \delta\underline{u}\|_{\dot{B}^{2\gamma-1}_{1, \infty}}
    + \|  \delta\underline{u}\|_{L^1(0, t; \dot{B}^{2\gamma-1}_{1, \infty})}\|  \underline{u_1}\|_{\dot{B}^{2\gamma-1}_{1, \infty}}\notag\\
      \lesssim&~   T_0 \|(\underline{u_1}, \underline{u_2}) \|_{L^\infty(0, T_0;   \dot{B}^{2\gamma-1}_{1, \infty})} \|\delta\underline{u}\|_{ L^\infty(0, T_0; \dot{B}^{2\gamma-1}_{1, \infty})},\label{U-esR4}
\end{align}
where in the last line we used  H\"older's inequality in time.

Finally, the term $R_3$  maybe handled along the same lines as previous and by induction on $\gamma\in\N$ and $\gamma>\gamma_c.$  More precisely, define
$$\mathcal{R}^\gamma_{1} (w):=   ((-\Delta_{\underline{u_2}})^\gamma-(-\Delta_{\underline{u_1}})^\gamma) w,\quad \mathcal{R}^\gamma_{2} (w):=   ((-\Delta_{y})^\gamma-(-\Delta_{\underline{u_1}})^\gamma) w,$$
we claim that we have
\begin{align}
     &\|\mathcal{R}^{\gamma}_1(w)\|_{\dot{B}^{-1}_{ 1, \infty}}\notag \\
      \lesssim& \, T_0 (1+ \|(\underline{u_1}, \underline{u_2}) \|_{L^\infty(0, T_0;   \dot{B}^{2\gamma-1}_{1, \infty})}) \|\delta\underline{u}\|_{ L^\infty(0, T_0; \dot{B}^{2\gamma-1}_{1, \infty})}\|w\|_{ L^\infty(0, T_0; \dot{B}^{2\gamma-1}_{1, \infty})} \label{U-esR31}
\end{align}
and
      \begin{align}
 \|\mathcal{R}^{\gamma}_2(w)\|_{\dot{B}^{-1}_{ 1, \infty}} 
      \lesssim   T_0 \|\underline{u_1} \|_{L^\infty(0, T_0;   \dot{B}^{2\gamma-1}_{1, \infty})} \|w\|_{ L^\infty(0, T_0; \dot{B}^{2\gamma-1}_{1, \infty})}.\label{U-esR32}
\end{align}
Notice that we have
\begin{align*}
   \mathcal{R}_1^\gamma(w)
    &= \left(-{\rm div}_y (    (\nabla_y w)   {\rm M}_{\underline{u_2}}{\rm M}_{\underline{u_2}}^{\rm T})\right)^\gamma -     \left(-{\rm div}_y (    (\nabla_y w)   {\rm M}_{\underline{u_1}}{\rm M}_{\underline{u_1}}^{\rm T})\right)^\gamma     ,\\
   \mathcal{R}_2^\gamma(w)&= (-{\rm div}_y  \nabla_y w)^\gamma -(-{\rm div}_y (    (\nabla_y w)   {\rm M}_{\underline{u_1}}{\rm M}_{\underline{u_1}}^{\rm T}))^\gamma,
\end{align*}
and 
\begin{align}\label{formula-dM}
    {\rm M}_{\underline{u_2}}{\rm M}_{\underline{u_2}}^{\rm T}-{\rm M}_{\underline{u_1}}{\rm M}_{\underline{u_1}}^{\rm T}=\delta{\rm M}({\rm Id}-{\rm M}_{\underline{u_2}}^{\rm T}) +({\rm Id}-{\rm M}_{\underline{u_1}} )\delta{\rm M}^{\rm T}-  \delta{\rm M}- \delta{\rm M}^{\rm T},
\end{align}
so it is easy to check  that \eqref{U-esR31} and \eqref{U-esR32} are satisfied when $\gamma=3.$ Now, assuming that they are satisfied for $\gamma-1$,  we write
\begin{multline*}
     ((-\Delta_{\underline{u_2}})^\gamma-(-\Delta_{\underline{u_1}})^\gamma)w
    \\=- ((-\Delta_{\underline{u_2}})^{\gamma-1}-(-\Delta_{\underline{u_1}})^{\gamma-1})\Delta_{\underline{u_2}}w 
    -(-\Delta_{\underline{u_1}})^{\gamma-1}(\Delta_{\underline{u_2}}-\Delta_{\underline{u_1}})w,
\end{multline*}
to find that
\begin{align*}
    &\| ((-\Delta_{\underline{u_2}})^\gamma-(-\Delta_{\underline{u_1}})^\gamma)w\|_{\dot{B}^{-1}_{1, \infty}}\\
    \leq&\, \|\mathcal{R}_1^{\gamma-1} (\Delta_{\underline{u_2}} w)\|_{\dot{B}^{-1}_{1, \infty}}+ \|(-\Delta_{\underline{u_1}})^{\gamma-1}(\Delta_{\underline{u_2}}-\Delta_{\underline{u_1}})w\|_{\dot{B}^{-1}_{1, \infty}}\\
    \lesssim& \, T_0 (1+\|(\underline{u_1}, \underline{u_2}) \|_{L^\infty(0, T_0;   \dot{B}^{2\gamma-3}_{1, \infty})}) \|\delta{\underline{u}}   \|_{ L^\infty(0, T_0; \dot{B}^{2\gamma-3}_{1, \infty})} \|\Delta_{\underline{u_2}} w   \|_{ L^\infty(0, T_0; \dot{B}^{2\gamma-3}_{1, \infty})} \\
    &\quad + \|(\Delta_{\underline{u_2}}-\Delta_{\underline{u_1}})w\|_{\dot{B}^{2\gamma-3}_{1, \infty}}\\
     \lesssim &\, T_0 (1+ \|(\underline{u_1}, \underline{u_2}) \|_{L^\infty(0, T_0;   \dot{B}^{2\gamma-1}_{1, \infty})}) \|\delta\underline{u}\|_{ L^\infty(0, T_0; \dot{B}^{2\gamma-1}_{1, \infty})}\|w\|_{ L^\infty(0, T_0; \dot{B}^{2\gamma-1}_{1, \infty})},  
\end{align*}
where in the last inequality we used the facts that the transform $X_{u_1}, X_{u_2}$ are volume preserving, and $\dot{B}^{2\gamma-2}_{1, \infty}$ is an algebra together with formula \eqref{formula-dM}. As such, we conclude  by induction that \eqref{U-esR31} is satisfied for all  $\gamma\in\N$ and $\gamma>\gamma_c.$  The proof of \eqref{U-esR32}  can be handled along the same lines.

With \eqref{U-esR31} and \eqref{U-esR32} in hand, we see that 
\begin{align}
    \|R_3\|_{\dot{B}^{-1}_{1, \infty}}\leq&~ \|\mathcal{R}_1^{\gamma}(\underline{u_2})\|_{\dot{B}^{-1}_{1, \infty}}+  \|\mathcal{R}_2^{\gamma}(\delta\underline{u})\|_{\dot{B}^{-1}_{1, \infty}}\notag\\
    \lesssim& ~ T_0 (1+ \|(\underline{u_1}, \underline{u_2}) \|_{L^\infty(0, T_0;   \dot{B}^{2\gamma-1}_{1, \infty})})^2 \|\delta\underline{u}\|_{ L^\infty(0, T_0; \dot{B}^{2\gamma-1}_{1, \infty})}.\label{U-esR3}
\end{align}
Note that owing to our assumptions on $(\underline{B_1}, \underline{u_1}, \underline{P1})$ and $(\underline{B_2}, \underline{u_2}, \underline{P_2}),$ all the factors involving the functions  $(\delta\underline{B}, \delta\underline{u}, \delta\underline{P})$ in the estimate  of $R_1, \cdots, R_4$ tend to $0$ when $T_0$ goes to $0$. Putting together inequalities \eqref{U-esR1}, \eqref{U-esR2}, \eqref{U-esR4} and \eqref{U-esR3} into inequality \eqref{U-es2}, and using \eqref{U-es1},   we get $(\delta\underline{B}, \delta\underline{u})\equiv 0$  on $[0, T_0]\times\T^d$ if $T_0$ is small enough. 
Finally,   uniqueness is proved on the whole  interval thanks to a standard connectivity  argument.

\subsection{The case \texorpdfstring{$\gamma\notin\N$}{TEXT}}
Let us  recall that, in $\R^d,$ for all $\sigma\in(0, 1)$ the  fractional Laplacian $(-\Delta)^{\sigma}=\Lambda^{2\sigma}$ enjoys a singular integral definition:
\begin{align}\label{Def-fracL}
   \Lambda^{2\sigma} f  = {\rm p. v.}\int_{\R^d} (f(x)-f(\underline{z}))\, K^{\sigma}(x-\underline{z})\,d\underline{z} \quad{\rm for}~~ x\in\R^d
\end{align}
with the kernel $K^\sigma(\underline{z})=\frac{c_{d, \sigma}}{|\underline{z}|^{d+2\sigma}},$ where $c_{d, \sigma}$ is a constant depending on dimension $d$  
and $\sigma.$  For periodic functions,
the representation is still valid due to  sufficient decay of $K^\sigma$ at infinity.   Indeed,  if $f$ is 1-periodic then the representation becomes 
\begin{align}\label{Def-fracL-P}
   \Lambda^{2\sigma} f  = {\rm p. v.}\int_{\T^d} (f(x)-f(\underline{z}))\, K^{\sigma}_{\rm per}(x-\underline{z})\,d\underline{z}\quad{\rm for}~~ x\in\T^d
\end{align}
with $K^\sigma_{\rm per}(\underline{z})=\sum_{{\rm j}\in\mathbb{Z}^d}\frac{c_{d, \sigma}}{|\underline{z}+{\rm j}|^{d+2\sigma}},$  see e.g.   \cite[(1.36)]{CS23}.
Thus in Lagrangian coordinates $ \Lambda^{2\sigma} u_i$ ($i=1,  2)$ recasts to
\begin{align*}
    \Lambda_{\underline{u_i}}^{2\sigma}{\underline{u_i}}&= {\rm p. v.}\int_{\T^d}   ({u_i}(X_{\underline{u_i}}(t, y))-{u_i}(\underline{z}))\, K^\sigma_{\rm per}(X_{\underline{u_i}}(t, y)-\underline{z})\,d\underline{z}\\
    &= {\rm p. v.}\int_{\T^d}  (\underline{u_i}(y)- \underline{u_i}(z))\, K^\sigma_{\rm per}(X_{\underline{u_i}}(t, y)- X_{\underline{u_i}}(t, z))\,dz
\end{align*}
for $ y\in\T^d,$ where we used the fact that the transform $X_{u_1}, X_{u_2}$ are volume preserving. The above notation has its limitation when considering difference estimates in the Lagrangian coordinates, so we are motivated to naturally extend its definition: Given any divergence free vector-valued function $v\in L^1(0, T; W^{1, \infty})$ and  $X_v$   defined by \eqref{def-trajectory},   $\forall \underline{f}: y\mapsto \underline{f}(y)$ we define
\begin{align}\label{Def-fracL'}
   \Lambda_{\underline{v}}^{2\sigma} \underline{f}: = {\rm p. v.}\int_{\T^d}  (\underline{f}(y)-\underline{f}(z)) \, K^\sigma_{\rm per}(X_{v}(t, y)-X_{v}(t, z))\,d{z}.
\end{align}
 
Now, since for any given $\gamma\notin\N$ and $\gamma>\gamma_c,$   $\gamma=  1+\theta+[\gamma] -1$ with $\theta=\gamma-[\gamma]\in (0, 1),$ and in Eulerian coordinates $(-\Delta)^\gamma= -{\rm div}\Lambda^{2\theta}\nabla (-\Delta)^{[\gamma]-1},$ thus one   finds by  definition \eqref{Def-fracL'} that  
\begin{align}
    (-\Delta_{\underline{u_i}})^\gamma{\underline{u_i}}
    =-{\rm div}_{\underline{u_i}}  \Lambda^{2\theta}_{\underline{u_i}}  \nabla_{\underline{u_i}} (-\Delta_{\underline{u_i}}  )^{[\gamma]-1}  \underline{u_i}.\label{formula-FracL}
\end{align}
Based on above formula, we are able to rewrite \eqref{MRE-Lag-d}  as
\begin{equation}\label{MRE-Lag-d'}
    \left\{\begin{aligned}
       \delta \underline{B}(t, y)&=   \int_0^t B_0(y)\cdot \nabla_y \delta\underline{u}(\tau, y)\,d\tau,\\
        (-\Delta_y)^{\gamma} \delta\underline{u}- \nabla_{y}\delta\underline{P}&=R_1+R_2+R_3',\\
        \mathrm{div}_y\delta\underline{u}&=  R_4,
    \end{aligned} \right.
\end{equation}
with
\begin{align*}
    R_3':=& \,({\rm div}_{\underline{u_1}}- {\rm div}_{\underline{u_2}}) \Lambda^{2\theta}_{\underline{u_2}}\nabla_{\underline{u_2}} (-\Delta_{\underline{u_2}})^{[\gamma]-1}\underline{u_2}\\
    &\quad+  {\rm div}_{\underline{u_1}}(\Lambda^{2\theta}_{\underline{u_1}}-\Lambda^{2\theta}_{\underline{u_2}}){\nabla}_{\underline{u_2}}(-\Delta_{\underline{u_2}})^{[\gamma]-1}\underline{u_2}\\
&\quad+  {\rm div}_{\underline{u_1}}\Lambda^{2\theta}_{\underline{u_1}}({\nabla}_{\underline{u_1}}-{\nabla}_{\underline{u_2}}) (-\Delta_{\underline{u_2}})^{[\gamma]-1}\underline{u_2}\\
    &\quad+ {\rm div}_{\underline{u_1}}\Lambda^{2\theta}_{\underline{u_1}}{\nabla}_{\underline{u_1}} ((-\Delta_{\underline{u_1}})^{[\gamma]-1}-(-\Delta_{\underline{u_2}})^{[\gamma]-1} ) {\underline{u_2}}\\
    &\quad+ ((-\Delta_{y})^{\gamma}+ {\rm div}_{\underline{u_1}} \Lambda^{2\theta}_{\underline{u_1}}\nabla_{\underline{u_1}} (-\Delta_{\underline{u_1}})^{[\gamma]-1})\delta\underline{u}\\
    =:& R_{31}'+R_{32}'+R_{33}'+R_{34}'+R_{35}'\cdotp
\end{align*}

In order to give an unified proof for $\theta\in(0, \infty),$  
we have to  perform estimates for $\delta\underline{u}$ in slightly larger Besov space compared to the case $\gamma\in\N$.   More precisely, define $p_*=d/(d-\theta)$ i.e. $\theta= d-d/p_*,$ and apply Lemma \ref{Le-L}
  gives that 
   \begin{multline}\label{U2-es2}
        \|\delta\underline{u}(t, \cdot)\|_{\dot{B}^{2\gamma-1-d+d/p_*}_{p_*, \infty}} +\| \delta\underline{P}\|_{\dot{B}^{-d +d/p_*  }_{p_*, \infty}}\\
        \lesssim \|(R_1, R_2, R_3')(t, \cdot)\|_{\dot{B}^{-1-d+d/p_*}_{p_*, \infty}} +\|R_4(t, \cdot)\|_{\dot{B}^{2\gamma-2-d+d/p_*}_{p_*, \infty}}.
   \end{multline}
   Moreover, using $2\gamma-1-d>1$ and thanks to the embedding property in Proposition \ref{P_Besov}, it is clear that 
   \begin{align}
   \|\delta \underline{B}(t, \cdot)\|_{L^2}\lesssim&\, \|B_0\|_{L^2}\| \nabla_y\delta\underline{u}\|_{L^1(0, t; L^\infty)}\notag\\
   \lesssim&\,
   \|B_0\|_{L^2}\| \delta\underline{u}\|_{L^1(0, t; {\dot{B}^{2\gamma-1-d+d/p_*}_{p_*, \infty}})}.\label{U2-es1}
    \end{align}

At this stage, let us remark that the estimates of the terms $R_1, R_2,  R_4$ in the   right-hand side of \eqref{U2-es2} are quite similar to what we did in the previous case. Indeed,  since $d/p_*-\theta>0,$ one has  the product law  (from Lemma \ref{Le-PL})
   \begin{align} \label{productlaw22}
  \dot{B}^{d/p_* }_{p_*,1} \times  \dot{B}^{ -\theta}_{p_*, \infty}   \hookrightarrow \dot{B}^{-\theta}_{p_*, \infty},
\end{align}
 and  the fact that the function space $\dot{B}^{2\gamma-2-d+d/p_*}_{p_*, \infty}$ is an algebra in our framework due to $2\gamma>2+d$ again.
 
 As such,  we will only focus on the estimate of $R_3'$ in the next.  
To do that, we claim that we have the following estimates, for which proof is presented   in Appendix \ref{A}.
\begin{lemma}\label{Le-es-FracL}
Let $0<s_1<1, ~s_1<s_2<1+s_1$ and $p_1, r_1\in[1, \infty].$ Assume that $\underline{v}, \underline{v_1}, \underline{v_2}$ satisfied \eqref{La-smallness}.   Let $w\in {\dot{B}^{s_2-s_1}_{p_1, r_1}}(\T^d),$ then it holds that
\begin{align}\label{Le-es-F1}
     \|\Lambda_{\underline{v}}^{s_2} w\|_{\dot{B}^{-s_1}_{p_1, r_1}} 
      \lesssim \,  \|w\|_{   \dot{B}^{s_2-s_1}_{p_1, r_1}}
\end{align}
and
      \begin{align}\label{Le-es-F2}
   \|(\Lambda_{\underline{v_1}}^{s_2}-\Lambda_{\underline{v_2}}^{s_2} )w\|_{\dot{B}^{-s_1}_{p_1, r_1}} 
      \lesssim \,  \|w\|_{   \dot{B}^{s_2-s_1}_{p_1, r_1}}\|\delta \underline{v}\|_{ L^1(0, T; W^{1, \infty})}.
\end{align}  
\end{lemma}

 Then,  thanks to product law \eqref{productlaw22} and Lemma \ref{Le-es-FracL}, and keeping in mind that $\theta=d-d/p_*,$ for  $C_{0}:=1+\|(\underline{u_1}, \underline{u_2})\|_{L^\infty(0, T_0; \dot{B}^{d+1}_{1, 1})}$ the  terms  in $R_3'$ can be estimated as follows, 
\begin{align*}
      \|R_{31}'\|_{\dot{B}^{-1-d+d/p_*}_{1, \infty}}\lesssim &\,  \|\delta{\rm M} \,\Lambda^{2\theta}_{\underline{u_2}}\nabla_{\underline{u_2}} (-\Delta_{\underline{u_2}})^{[\gamma]-1}\underline{u_2}\|_{\dot{B}^{-\theta}_{1, \infty}}\\
      \lesssim &\,  
      \|\delta\underline{u}\|_{L^1(0, T_0; \dot{B}^{d/p_* +1}_{p_*, 1})} \|\Lambda^{2\theta}_{\underline{u_2}}\nabla_{\underline{u_2}} (-\Delta_{\underline{u_2}})^{[\gamma]-1}\underline{u_2}\|_{L^\infty(0, T_0; \dot{B}^{-\theta}_{p_*, \infty})}\\
        \lesssim &\, C_0\, T_0 \,\|\delta\underline{u}\|_{L^1(0, T_0; \dot{B}^{d/p_*+1}_{p_*, 1})}\| \nabla_{\underline{u_2}}(-\Delta_{\underline{u_2}})^{[\gamma]-1}\underline{u_2}\|_{L^\infty(0, T_0; \dot{B}^{\theta}_{q*, \infty})}\\
\lesssim &\,C_0 \,T_0 \,\|\delta\underline{u}\|_{L^\infty(0, T_0; \dot{B}^{d/p_*+1}_{p_*, 1})}\|\underline{u_2}\|_{L^\infty(0, T_0; \dot{B}^{2[\gamma]-1+\theta}_{p_*, \infty})}
\end{align*}
and
\begin{align*}
      \|R_{32}'\|_{\dot{B}^{-1-d+d/p_*}_{p_*, \infty}}\lesssim &\,  C_0\, T_0\, \|(\Lambda^{2\theta}_{\underline{u_1}}-\Lambda^{2\theta}_{\underline{u_2}})\nabla_{\underline{u_2}}(-\Delta_{\underline{u_2}})^{[\gamma]-1}\underline{u_2}\|_{L^\infty(0, T_0; \dot{B}^{-\theta}_{p_*, \infty})} \\
\lesssim&\, C_0\, T_0\, \|\delta\underline{u}\|_{L^1(0, T_0; {W}^{1, \infty})} \|\nabla_{\underline{u_2}}(-\Delta_{\underline{u_2}})^{[\gamma]-1}\underline{u_2}\|_{L^\infty(0, T_0; \dot{B}^{\theta}_{p_*, \infty})}\\
\lesssim&\, C_0\, T_0 \,\|\delta\underline{u}\|_{L^\infty(0, T_0; {W}^{1, \infty})} \|\underline{u_2}\|_{L^\infty(0, T_0; \dot{B}^{2[\gamma]-1+\theta}_{p_*, \infty})}.
\end{align*}
Furthermore, noticing that  we have
\begin{align*}
    \dot{B}^{d/p_* }_{p_*,1} \times  \dot{B}^{ \theta}_{p_*, \infty}   \hookrightarrow \dot{B}^{\theta}_{p_*, \infty}\quad ({\rm since }~\theta-d/p_*<0),
\end{align*}
thus
\begin{align*}
\|R_{33}'\|_{\dot{B}^{-1-d+d/p_*}_{p_*, \infty}}\lesssim &\,  C_0\, T_0\, \|\Lambda^{2\theta}_{\underline{u_1}}({\nabla}_{\underline{u_1}}-{\nabla}_{\underline{u_2}}) (-\Delta_{\underline{u_2}})^{[\gamma]-1}\underline{u_2}\|_{\dot{B}^{-\theta}_{p_*, \infty}}\\
\lesssim &\,  C_0\, T_0\,  \|({\nabla}_{\underline{u_1}}-{\nabla}_{\underline{u_2}}) (-\Delta_{\underline{u_2}})^{[\gamma]-1}\underline{u_2}\|_{\dot{B}^{\theta}_{p_*, \infty}}\\
\lesssim &\,  C_0\, T_0\,  \|\delta\underline{u}\|_{L^1(0, T_0; \dot{B}^{d/p_*+1}_{p_*, 1})} \| \nabla_y (-\Delta_{\underline{u_2}})^{[\gamma]-1}\underline{u_2}\|_{\dot{B}^{\theta}_{p_*, \infty}}\\
\lesssim &\,  C_0\, T_0\,  \|\delta\underline{u}\|_{L^\infty(0, T_0; \dot{B}^{d/p_*+1}_{p_*, 1})} \| \underline{u_2}\|_{\dot{B}^{2[\gamma]-1+\theta}_{p_*, \infty}}.
\end{align*}
And,  similar to \eqref{U-esR31} (using $2[\gamma]-1+\theta =2\gamma-1-d+d/p_*> d/p_*+1$)
\begin{align*}
 \|R_{34}'\|_{\dot{B}^{-1-d+d/p_*}_{p_*, \infty}}\lesssim &\,  C_0\, T_0\, 
 \|\Lambda^{2\theta}_{\underline{u_1}}{\nabla}_{\underline{u_1}} ((-\Delta_{\underline{u_1}})^{[\gamma]-1}-(-\Delta_{\underline{u_2}})^{[\gamma]-1} ) {\underline{u_2}}\|_{\dot{B}^{-\theta}_{p_*, \infty}}\\
\lesssim &\,  C_0\, T_0\,  \| ((-\Delta_{\underline{u_1}})^{[\gamma]-1}-(-\Delta_{\underline{u_2}})^{[\gamma]-1} ) {\underline{u_2}}\|_{\dot{B}^{1+\theta}_{p_*, \infty}}\\
\lesssim&\, C_0^2\, T_0\,  \|\delta\underline{u}\|_{L^\infty(0, T_0; \dot{B}^{2[\gamma]-1+\theta}_{p_*, \infty})} \| \underline{u_2}\|_{\dot{B}^{2[\gamma]-1+\theta}_{p_*, \infty}}.
\end{align*}
The last term $R'_{35}$ may be handled along the same lines. Indeed, we have
\begin{align*}
R_{35}'=  -{\rm div}_y\Lambda_y^{2\theta}\nabla (-\Delta_y)^{[\gamma]-1} + {\rm div}_{\underline{u_1}} \Lambda^{2\theta}_{\underline{u_1}}\nabla_{\underline{u_1}} (-\Delta_{\underline{u_1}})^{[\gamma]-1})\delta\underline{u}.
\end{align*}

Finally, putting together all the previous inequalities into \eqref{U2-es2} and using  \eqref{U2-es1},  we
  conclude uniqueness  on a sufficiently small time interval, and then on the whole interval $\R_+.$
\bigbreak

Let us now explain how the arguments have to be modified to prove the uniqueness part of Theorem \ref{Th2}. In fact, one only need to consider the case that $\eps$ is small enough (say $\eps<1/3$), and then follow the above proof with small modification.
In particular, in the case $\gamma\in \N,$ one may perform estimates for the equation of $\delta\underline{u}$ in $\dot{B}^{-1}_{p, \infty},$ 
while for the case $\gamma\notin\N$ one may choose   the corresponding space to be $\dot{B}^{-1-d/p+d/p_*}_{p_*, \infty}$ with $p_*=2dp/(2d-p\theta)$ i.e. $\theta/2= d/p-d/{p_*}.$  The details are left to the interested reader.

We complete the proof of uniqueness in Theorem \ref{Th2-1} and Theorem \ref{Th2} . \qed
\bigbreak

\begin{proof}[Proof of Corollary \ref{C0}]
We remark that the argument is the same for  \eqref{behavior-u1} and \eqref{behavior-u2}.
  Applying the time derivative to the velocity equation and using the equality of \eqref{eq-dtbb}, we get
\begin{align}\label{eq-dtu}
   (-\Delta)^{\gamma} \partial_t u=  \mathbb{P}\div ( \nabla u(B\otimes B) +(B\otimes B) (\nabla u)^{\rm T}-  \div(u\odot (B\otimes B))  ).
\end{align}
Observing from Lemma \ref{Le-PL} the following product law:
\begin{align*}
    \dot{B}^{d-1}_{1, 1}\times \dot{B}^{0}_{1, \infty}\hookrightarrow  \dot{B}^{-1}_{1, 1}
\end{align*}
and recalling the product law\eqref{productlaw1},
it yields
\begin{align*}
   \| \mathbb{P}\div ( \nabla u(B\otimes B) +(B\otimes B) (\nabla u)^{\rm T})\|_{\dot{B}^{-2}_{1, 1}}\lesssim&~   \|\nabla u(B\otimes B)\|_{\dot{B}^{-1}_{1, 1}}\\
   \lesssim&~   \|\nabla u\|_{\dot{B}^{d-1}_{1, 1}}\|B\otimes B\|_{\dot{B}^{0}_{1, \infty}}\\
    \lesssim&~   \|  u\|_{\dot{B}^{d}_{1, 1}}\|B\|_{L^2}^2,
\end{align*}
and  
\begin{align*}
    \| \mathbb{P}\div \div(u\odot (B\otimes B))\|_{\dot{B}^{-2}_{1, 1}}\lesssim&~ \|u\odot (B\otimes B))\|_{\dot{B}^{0}_{1, 1}}\\
    \lesssim&~ \|u\|_{\dot{B}^{d}_{1, 1}}
    \|B\otimes B\|_{\dot{B}^{0}_{1, \infty}}\\
 \lesssim&~ \|u\|_{\dot{B}^{d}_{1, 1}}
    \|B\|_{L^2}^2.
\end{align*} 
Taking account of above two inequalities and $\gamma\geq \gamma_c,$ we see from  \eqref{eq-dtu} that
\begin{align*}
    \|\partial_t u\|_{L^\infty(\R_+; \dot{B}^{2\gamma-2}_{1, 1})}\lesssim& \|u\|_{L^\infty(\R_+; \dot{B}^{d}_{1, 1})}
    \|B\|_{L^\infty(\R_+; L^2)}^2\\
    \lesssim& \|u\|_{L^\infty(\R_+; \dot{B}^{2\gamma-1}_{1, \infty})}
    \|B_0\|_{L^2}^2.
\end{align*}
This, together with the embedding $\dot{B}^{2\gamma-2}_{1, 1}\hookrightarrow \dot{H}^{2\gamma-2-d/2}$
and $\gamma\geq \gamma_c$ again  imply that  $\partial_t u\in L^\infty(\R_+; H^{d/2}).$
Now, noticing that $ u\in L^2(\R_+; {H}^{d/2})$ and 
\begin{align*}
    \|u(t_1, \cdot)-u(t_2, \cdot)\|_{H^{d/2}}\leq&  \left|\int_{t_1}^{t_2}\|\partial_t u(t, \cdot)\|_{H^{d/2}}\,dt\right|\\
   \leq&   \|\partial_t u\|_{L^\infty(\R_+; H^{d/2})} \,|t_1-t_2|.
\end{align*}
we get  $\|u(t, \cdot)\|_{H^{d/2}}\to 0$ as $t\to\infty$ by \cite[Lemma 3.1]{DWZZ}, which assesses that a uniformly continuous and integrable function must vanish at infinity. Finally, \eqref{behavior-u1} and \eqref{behavior-u2} are proved by interpolating with  the $\dot{B}^{2\gamma-1}_{1, \infty}$ norm of $u,$ respectively, which is uniformly bounded in time thanks to Remark \ref{re-solu}.
\end{proof}


 \appendix
 \section{Function spaces and useful inequalities}\label{A}
Here, we recall the homogeneous Littlewood-Paley decomposition for periodic functions, the definition of  Besov space on torus $\T^d$ and some useful properties.  For  reference, we refer to  \cite[Appendix A]{CS23}.
In fact the analysis for periodic case is essentially the same with  the whole-space case,  and the   details   for the latter can be found in e.g. \cite{Ba11}.  

Let  $ {\varphi}\in\mathcal{D}(\mathcal{A})$ be a smooth function supported in the annulus   $\mathcal{A}=\{\xi\in\mathbb{R}^d, \,\frac34\leq |\xi|\leq \frac83\}$ and such that 
\begin{gather*}
\sum_{j\in\mathbb{Z}}  {\varphi}(2^{-j}\xi)=1,\quad \forall \xi\in\mathbb{R}^d    \backslash \{0\}.
\end{gather*}
For any $u\in \mathcal{S}'(\mathbb{T}^d)$, we have the Fourier series representation
\begin{equation*}
    u(x)=\sum_{{\rm k}\in\Z^d} \widehat{u}_{\rm k} \,e^{i{\rm k}\cdot x}\quad{\rm with}\quad \widehat{u}_{\rm k} := \dfrac{1}{|\T^d|}\int_{\T^d} u(x)\,e^{-i{\rm k}\cdot x}\,dx.
\end{equation*}
Note that in the sequel of this section our functions will always have mean
zero. 
The periodic dyadic block   $\dot{\Delta}_j$ and   the low frequency cut-off $\dot{S}_j$ are defined by 
\begin{equation*}
 \forall j \in \mathbb{Z},\quad \dot{\Delta}_ju:=\sum_{{\rm k}\in\Z^d}{\varphi}(2^{-j} {\rm k})\,\widehat{u}_{\rm k} \,e^{i{\rm k}\cdot x} \quad\text{and}\quad\dot{S}_j u:=\sum_{\ell \leq j-1 }\dot{\Delta}_{\ell} u.
\end{equation*}
Then we have the  decomposition
\begin{equation*}
u=\sum_{j\in\mathbb{Z}}\dot{\triangle}_ju, \quad \forall\, u \in{\mathcal S}'({\mathbb T}^d).
\end{equation*}
Moreover, the homogeneous Littlewood-Paley decomposition satisfies the property of almost orthogonality:
\begin{equation*}
{\dot\Delta}_j{\dot\Delta}_{j'} u=0, \quad {\rm{if}} \ |j-j'|\geq 2, \quad \dot\Delta_j(S_{j'-1}u \dot\Delta_j' u)=0, \quad {\rm{if}} \ |j-j'|\geq 5,
\end{equation*}
and it is obvious that $\dot{\Delta}_j u\equiv 0$ for negative enough $j,$ i.e. low frequency vanishes.

We now recall the definition of homogeneous Besov spaces.
\begin{definition}
	Let $s$ be a real number and $(p, r)$ be in $[1,\infty]^{2}$, we set
	\begin{equation*}
	\|u\|_{\dot B^{s}_{p, r}(\T^d) }  :=  \left\{
	\begin{split}
	&\|2^{js}\|\dot\Delta_{j}u\|_{L^{p} (\T^d)}\|_{\ell^r(\mathbb{Z})} \quad\, \,~~\quad{\rm{for}} ~ 1 \leq r < \infty,\\
	& \sup_{j \in \mathbb{Z}} 2^{j s}\|\dot\Delta_{j}u\|_{L^{p} (\T^d)}~\,\,\quad\qquad {\rm{for}} ~  r= \infty.
	\end{split}
	\right.
	\end{equation*}
	The  homogeneous Besov space $\dot B^{ s}_{p, r}(\T^d):= \{ u \in \mathcal{S}':	\|u\|_{\dot B^{s}_{p, r}(\T^d)}<\infty \}$.
\end{definition}
It is clear that $\|\cdot\|_{\dot{H}^{s}(\T^d)}=\|\cdot\|_{\dot{B}^{s }_{2, 2}(\T^d)},$ where $\dot{H}^s(\T^d)$ is the standard homogeneous Sobolev spaces.

We will now give a characterization of Besov spaces  in terms of finite differences,  see e.g.  \cite[(1.22)]{CS23} or  \cite[Theorem 2.36]{Ba11}.
\begin{proposition}\label{P-F-B}
Let $s\in(0, 1)$ and $p, r\in[1, \infty].$ Then for any $u\in \dot{B}^s_{p, r}(\T^d),$
\begin{align*}
  \left\|\frac{\|u(\cdot+y)-u(\cdot)\|_{L^p(\T^d)}}{|y|^s} \right\|_{L^r\left(\T^d; \frac{dy}{|y|^d}\right)}\sim\|u\|_{\dot{B}^s_{p, r}(\T^d)}.
\end{align*}

\end{proposition}

Next, we recall some basic facts on the Besov spaces.

\begin{proposition}\label{P_Besov}
	Let $1\leq p \leq \infty$. Then there hold:
	
	\begin{itemize}
		\item for all $s\in\mathbb R$~~and~~$1\leq p, r\leq\infty,$ 
		we have 
			\begin{equation*}
		\|D^{k}u\|_{\dot B^{s}_{p, r}(\T^d)}\sim\|u\|_{\dot B^{s+k}_{p, r}(\T^d)}.
		\end{equation*}
		\item for any $\theta\in(0, 1)$ and $\,  \underline{s}<\bar  s,$ we have
$$\|u\|_{\dot B^{\theta   \underline{ s}+(1-\theta)\bar{s}}_{p, 1}(\T^d)}\lesssim\|u\|_{\dot B^{  \underline {s}}_{p, \infty}(\T^d)}^{\theta}\|u\|_{\dot B^{\bar{s}}_{p, \infty}(\T^d)}^{1-\theta}.$$
        \item Embedding: we have the following continuous embedding
$$\dot B^{s}_{p,   r}(\T^d)\hookrightarrow \dot B^{s-\frac{d}{p}}_{\infty, \infty}(\T^d)\quad{\rm{ whenever }} \ 1\leq p,   r\leq\infty,$$
and
$$\dot{B}^0_{p, 1}(\T^d)\hookrightarrow L^p(\T^d) \hookrightarrow\dot{B}^0_{p, \infty}(\T^d),$$

$$\dot{B}^{\bar{s}}_{p, \infty}(\T^d) \hookrightarrow\dot{B}^{\underline{s}}_{p, 1}(\T^d) \quad{\rm{ whenever }} ~~\underline{s}<\bar{s}.$$

\item  Let $f$ be a bounded function on $\mathbb{Z}^{d}\setminus \{0\}$ which is homogeneous of degree 0.  Define $f(D)$ on $\cS(\T^d)$  by \begin{equation*}
 (f(D)u)(x){:=}\sum_{{\rm k}\in\Z^d} f({\rm k})\, \widehat{u}_{\rm k}\,e^{i{\rm k}\cdot x}.
\end{equation*}
Then,  for all exponents $(s,p,r),$ we have the estimate
$$\|f(D) u\|_{\dot B^s_{p,r}(\T^d)}\lesssim \|u\|_{\dot B^s_{p,r}(\T^d)}.$$
	\end{itemize}
\end{proposition}

The following logarithmic interpolation inequality is available.
\begin{lemma}\cite[Remark 2.105]{Ba11}\label{lemma-log}
There exists a constant $C$   such that for all $s\in\mathbb{R},$ $\theta>0$ and $1\leq p\leq \infty,$ we have
\begin{equation}\label{in-log}
\|f\|_{\dot{B}^{s}_{p, 1}(\T^d)}\leq C \|f\|_{\dot{B}^s_{p, \infty}(\T^d)} \left( 1+\log \left(\dfrac{\|f\|_{\dot{B}^{s+\theta}_{p, \infty}(\T^d)}}{ \|f\|_{\dot{B}^s_{p, \infty}(\T^d)} }\right) \right).
\end{equation}
\end{lemma}
\medskip
 The following Kato-Ponce type inequalities can be found in \cite[Corollary 5.2]{Li19}.
 \begin{lemma}\label{Le-KP}
Let~$s>0$ and  $1<p<\infty.$ Then for any $f, g\in \mathcal{S}(\T^d),$
\begin{equation*}
\|[\Lambda^s, f]g\|_{L^p (\T^d)}\leq C(\|\Lambda^s f\|_{L^{p_1}(\T^d) }\|g\|_{L^{p_2}(\T^d) }+\|\nabla f \|_{L^{p_3}(\T^d) }\|\Lambda^{s-1}g\|_{L^{p_4}(\T^d) }),
\end{equation*}
where $1<  p_1,\, p_2,\, p_3,\, p_4 \leq \infty,$  and $\frac{1}{p}=\frac{1}{p_1}+\frac{1}{p_2}=\frac{1}{p_3}+\frac{1}{p_4}.$
     \end{lemma}
     
\begin{lemma}\cite[Corollary 2.54]{Ba11}\label{Le-KP2}
Let $s>0$ and $p, r\in[1, \infty].$ Then $\dot{B}^s_{p, r} \cap L^\infty$ is an algebra and we have
  \begin{equation*}
\|fg\|_{\dot{B}^s_{p, r}(\T^d)}\leq C(\| f\|_{\dot{B}^s_{p, r}(\T^d)}\|g\|_{L^\infty(\T^d)}+\|f\|_{L^\infty(\T^d)}\|g \|_{\dot{B}^s_{p, r}(\T^d)}).
\end{equation*}
 \end{lemma}
We  frequently used the following product law (see e.g. \cite[Proposition 2.1]{AGZ}).
\begin{lemma}\label{Le-PL}
Let $1\leq  p<\infty$ and $s_1+s_2>0$ with $s_1\leq d/p,~s_2<d/p,$ then 
    \begin{align} 
   \dot{B}^{s_1}_{p,1}(\T^d)\times  \dot{B}^{ s_2}_{p, \infty}  (\T^d) \hookrightarrow \dot{B}^{ s_1+s_2-d/p}_{p, 1}(\T^d).
\end{align}
\end{lemma}

Next, we  need  the regularity estimates in Besov spaces for the fractional Stokes equation:
\begin{equation}\label{eq-Stokes}
   \left\{ \begin{aligned}
     (-\Delta)^\gamma u- \nabla P&= \div F\quad {\rm on}~~\T^d,\\
     \div u&=g.
 \end{aligned}\right.
\end{equation}
 
\begin{lemma}\label{Le-L}
Let $\gamma\geq 0,$ $1\leq p, r\leq\infty$ and $s\in\R.$  Assume that $F\in\dot B^{s}_{p, r}(\T^d).$ Then, \eqref{eq-Stokes} has a unique  solution $(u,   P)$ in $\dot B^{2\gamma+s-1}_{p, r}(\T^d)\times \dot B^{s}_{p, r}(\T^d)$  and there exists a constant $C$ depending only on $d$  such that 
\begin{equation*} 
\|u\|_{\dot B^{2\gamma+s-1}_{p, r}(\T^d)} +\|P\|_{\dot{B}^s_{p, r}(\T^d)}\leq C(\|F\|_{\dot B^{s}_{p, r}(\T^d)}+ \|g\|_{\dot{B}^{2\gamma+s-2}_{p, r}(\T^d)}).
\end{equation*}
\end{lemma}
\begin{proof}
  The proof of the case $g\equiv0$ is classical. The general case follows from this particular case by  considering $v:=u+\nabla(-\Delta)^{-1}g,$ which satisfies
 \begin{equation*} 
   \left\{ \begin{aligned}
     (-\Delta)^\gamma v- \nabla (P+(-\Delta)^{\gamma-1}g)&= \div F\quad {\rm on}~~\T^d,\\
     \div v&=0.
 \end{aligned}\right.
\end{equation*}
\end{proof}

In the last of this section, we give the proof of Lemma \ref{Le-es-FracL}. 
\begin{proof}[Proof of Lemma \ref{Le-es-FracL}]
    The proof is via duality argument.  We first prove \eqref{Le-es-F1}.
    For any $f\in \mathcal{S}(\T^d)\cap \dot{B}^{s_1}_{p_{1}', r_{1}'}(\T^d),$   we   use  the periodic representation of the kernel $K^{s_2}$ to write
    \begin{align*}
   &\langle \Lambda^{s_2}_{\underline{v}} w, f\rangle\\
   =&~{\rm p.v.}\int_{\T^d} \int_{\T^d}f(y)(w(y)-w(z)) \,K^{\frac{s_2}{2}}_{\rm per}(X_{\underline{v}}(t, y)-X_{\underline{v}}(t, z))\,dzdy\\
    =&~\frac{1}{2} {\rm p.v.}\int_{\T^d} \int_{\T^d}(f(y)-f(z))(w(y)-w(z)) \,K^{\frac{s_2}{2}}_{\rm per}(X_{\underline{v}}(t, y)-X_{\underline{v}}(t, z))\,dzdy\\
     =&~\frac{1}{2} {\rm p.v.}\int_{\T^d} \int_{\T^d}(f(y)-f(z+y))(w(y)-w(z+y)) \\
     &\hspace*{6cm} \cdot K^{\frac{s_2}{2}}_{\rm per}(X_{\underline{v}}(t, y)-X_{\underline{v}}(t, z+y))\,dzdy.
    \end{align*}
    Note that  in our framework
    \begin{align*}
    |X_{\underline{v}}(t, y)-X_{\underline{v}}(t, z+y)|\sim  |z|.
    \end{align*}
    Indeed, we have
    \begin{align}\label{A4}
     X_{\underline{v}}(t, y)-X_{\underline{v}}(t, z+y) =-z- z\cdot \int_0^1\int_0^t (\nabla_y\underline{v})(\tau, y+ az )\,d\tau da.
    \end{align}
Hence by H\"{o}lder's inequality and  Proposition \ref{P-F-B} (note that our assumptions ensure $0<s_2,\, s_2-s_1<1$), and the fact that $w,  f$ are periodic,
     \begin{align*}
    &|\langle \Lambda^{s_2}_{\underline{v}} w, f\rangle|\\
        \lesssim&   ~{\rm p.v.}\int_{\T^d}        \|w(\cdot)-w(z+\cdot)\|_{L^{p_1}} \|f(\cdot)-f(z+\cdot) \|_{L^{p_1'}}  K^{\frac{s_2}{2}}_{\rm per}(z)\,dz   \\
           \lesssim&   ~{\rm p.v.}\int_{\T^d}        \|w(\cdot)-w(z+\cdot)\|_{L^{p_1}} \|f(\cdot)-f(z+\cdot) \|_{L^{p_1'}}  \frac{c_{d,  \frac{s_2}{2}}}{|z|^{d+s_2}}       \,dz   \\
             &\quad +  \sum_{{\rm j}\in\mathbb{Z}^d,  |{\rm j}|\geq 2\sqrt{d}}\int_{\T^d}        \|w(\cdot)-w(z+\cdot)\|_{L^{p_1}} \|f(\cdot)-f(z+\cdot) \|_{L^{p_1'}}  \frac{c_{d,  \frac{s_2}{2}}}{|z+{\rm j}|^{d+s_2}}       \,dz  \\
    \lesssim&~ {\rm p.v.}\int_{\T^d} \frac{\|w(\cdot+z)-w(\cdot)\|_{L^{p_1}(\T^d)}}{|z|^{s_2-s_1}}  \frac{\|f(\cdot+z)-f(\cdot)\|_{L^{p_1'}(\T^d)}}{|z|^{s_1}}   \frac{c_{d,  \frac{s_2}{2}}}{|z|^{d}} \,dz  \\
     \lesssim&~ \| w\|_{\dot{B}^{s_2-s_1}_{p_{1}, r_{1}}(\T^d)}~\|f\|_{\dot{B}^{s_1}_{p_{1}', r_{1}'}(\T^d)} 
    \end{align*}
    for $\frac{1}{p_1'}+ \frac{1}{p_1}=  \frac{1}{r_1'}+\frac{1}{r_1}=1,$
where we used the fact that 
\begin{align}\label{esA-remainder}
 \left\|\sum_{{\rm j}\in\mathbb{Z}^d,  |{\rm j}|\geq 2\sqrt{d}}  \frac{c_{d,  \frac{s_2}{2}}}{|\cdot+{\rm j}|^{d+s_2}} \right \|_{L^\infty(\T^d)}\lesssim  \sum_{{\rm j}\in\mathbb{Z}^d,  |{\rm j}|\geq 1} \frac{c_{d,  \frac{s_2}{2}}}{|{\rm j}|^{d+s_2}} <\infty.
\end{align}

    Similarly,   to prove \eqref{Le-es-F2} we  can write
    \begin{align*}
     &\langle (\Lambda^{s_2}_{\underline{v1}}- \Lambda^{s_2}_{\underline{v2}}) w, f\rangle\\
 =&~{\rm p.v.}\int_{\T^d} \int_{\T^d}f(y)(w(y)-w(z)) \\
& \qquad \qquad \qquad \cdot\left( K^{s_2}_{\rm per}(X_{\underline{v_1}}(t, y)-X_{\underline{v_1}}(t, z)) -K^{s_2}_{\rm per}(X_{\underline{v_2}}(t, y)-X_{\underline{v_2}}(t, z))\right)\,dzdy\\
 =&~\frac{1}{2}  {\rm p.v.}\int_{\T^d} \int_{\T^d}(f(y)-f(z))(w(y)-w(z)) \\
& ~\qquad \qquad \qquad \cdot\left( K^{s_2}_{\rm per}(X_{\underline{v_1}}(t, y)-X_{\underline{v_1}}(t, z)) -K^{s_2}_{\rm per}(X_{\underline{v_2}}(t, y)-X_{\underline{v_2}}(t, z))\right)\,dzdy\\
 =&~\frac{1}{2}  {\rm p.v.}\int_{\T^d} \int_{\T^d}(f(y)-f(z+y))(w(y)-w(z+y)) \\
&~ \quad  \qquad \cdot\left( K^{s_2}_{\rm per}(X_{\underline{v_1}}(t, y)-X_{\underline{v_1}}(t, z+y)) -K^{s_2}_{\rm per}(X_{\underline{v_2}}(t, y)-X_{\underline{v_2}}(t, z+y))\right)\,dzdy.
    \end{align*}
 And, thanks to \eqref{A4} we have
    \begin{align*}
   &| (X_{\underline{v_1}}(t, y)-X_{\underline{v_1}}(t, z+y)) - (X_{\underline{v_2}}(t, y)-X_{\underline{v_2}}(t, z+y))|\\
   \lesssim& ~\frac{\|\nabla_y \delta\underline{v}\|_{L^1(0, t; L^\infty(\T^d))}}{|z|^{d+s_2}}.
    \end{align*}
    Thus, applying H\"{o}lder's inequality and \eqref{esA-remainder} again yields \eqref{Le-es-F2}.
\end{proof}
 

\section*{Acknowledgement}
The author would like to thank the three  referees for their careful reading,  helpful suggestions and valuable comments,  which helped   to improve the presentation of  the manuscript a lot.  
 This work was partially supported by   the CY Initiative of Excellence (project CYNA),   the Direct Grant  of CUHK (No.  4053715)  and  the Hong Kong RGC grant (CUHK14302525).  
\bigbreak

{\bf Data Availibility.}
Data sharing is not applicable to this article as no data sets were generated or analyzed during the current study.
 


\begin{thebibliography}{10}

\bibitem{AGZ}\textsc{Abidi,H., Gui,G., Zhang,P.}: Well-posedness of 3-D inhomogeneous Navier-Stokes equations with highly oscillating initial velocity field. {\emph{J. Math. Pures Appl.}}   {\bf 100}(2),   166--203 (2013) 


 
\bibitem{Ar66} \textsc{Arnold, V.I.}: Sur la géométrie différentielle des groupes de Lie de dimension infinie et ses applications à l’hydrodynamique des fluides parfaits. {\emph {Ann. Inst. Fourier (Grenoble)}} {\bf 16}(fasc. 1), 319--361 (1966)
\bibitem{Ar74} \textsc{Arnold, V.I.}: The asymptotic Hopf invariant and it applications. In: Proceedings of Summer School in Differential Equations. Armenian Academy of Sciences, English translation, vol. {\bf 561} (1974)

\bibitem{AK98} \textsc{Arnold, V.I., Khesin, B.A.}: Topological methods in hydrodynamics. In: Applied Mathematical Sciences, vol. {\bf 125}. Springer, New York (1998)


\bibitem{BKS23} \textsc{Bae, H., Kwon, H., Shin, J.:} {Global solutions to Stokes-Magneto equations with fractional dissipations}.  arXiv:2310.03255 (2023)


\bibitem{Ba11} \textsc{Bahouri, H.,  Chemin,  J. Y., Danchin, R.}:  { Fourier Analysis and Nonlinear Partial Differential Equations}.   Grundlehren der Mathematischen Wissenschaften, Vol.  {\bf 343}.  Springer,  Berlin (2011)

\bibitem{BFV} \textsc{Beekie, R.,   Friedlander, S.,   Vicol, V.}: {On Moffatt’s magnetic relaxation equations}. \textit{Commun. Math. Phys.}
{\bf 390}(3), 1311--1339 (2022)


 \bibitem{BF13}\textsc{Boyer F.  and  Fabrie P.}: Mathematical tools for the study of the incompressible Navier-Stokes equations and related models. Vol.{\bf  183}. Applied Mathematical Sciences. Springer, New York (2013)


\bibitem{Br14} \textsc{Brenier, Y.}: Topology-preserving diffusion of divergence-free vector fields and magnetic relaxation. \textit{ Commun. Math. Phys.} {\bf 330}(2), 757--770 (2014)


\bibitem{CS23}\textsc{Cameron, S., Strain, R.M.}:  Critical local well-posedness for the fully nonlinear Peskin problem. \textit{Comm. Pure Appl. Math.} (2023)


 \bibitem{C16} \textsc{Chemin, J.,    McCormick, D.,    Robinson, J.,  Rodrigo, J.}:  Local existence for the non-resistive MHD equations in Besov spaces.  \textit{Adv. Math.} {\bf  286},  1--31  (2016)



\bibitem{CP23} \textsc{Constantin, P., Pasqualotto, F.}: Magnetic Relaxation of a Voigt–MHD System. \textit{ Commun. Math. Phys.} {\bf 402}(2), 1931--1952 (2023)


\bibitem{DGL}\textsc{Dalibard, A.-L.,   Guillod, J.,  Leblond, A.}:
Long-time behavior of the Stokes-transport system in a channel. \textit{arXiv:2306.00780v1}
 
 
\bibitem{DM1} \textsc{Danchin R.,   Mucha P.B.}: A Lagrangian approach for the incompressible Navier-Stokes equations with variable density. \textit{ Comm. Pure Appl. Math.} {\bf 65},   1458--1480 (2012)

\bibitem{DM2} \textsc{Danchin R.,   Mucha P.B.}: Incompressible flows with piecewise constant density. \emph{Arch. Ration. Mech. Anal.} 
{\bf 207},  991--1023 (2013)



\bibitem{DWZZ}\textsc{Doering, C.R., Wu, J., Zhao, K., Zheng, X.}: Long time behavior of the two- dimensional Boussinesq equations without buoyancy diffusion. \textit{Physica D} {\bf 376}(377), 144--159 (2018)



\bibitem{EM70} \textsc{Ebin, D.G., Marsden, J.E.}: Groups of diffeomorphisms and the motion of an incompressible fluid. \textit{ Ann. Math.} {\bf 2}(92), 102--163 (1970)


\bibitem{EM20} \textsc{Elgindi, T.M., Masmoudi, N.}: \texorpdfstring{$L^\infty$}{TEXT} ill-posedness for a class of equations arising in hydrodynamics. \emph{Arch. Ration. Mech. Anal.} {\bf 235}(3), 1979--2025 (2020)


 \bibitem{F14} \textsc{Fefferman, C., McCormick, D., Robinson, J.,   Rodrigo, J.}:  Higher order commutator estimates and local existence for the non-resistive MHD equations and related models. \textit{J. Funct. Anal.} {\bf 267}, 1035--1056 (2014)


 \bibitem{F17} \textsc{Fefferman, C., McCormick, D., Robinson, J.,   Rodrigo, J.}: Local existence for the non-resistive MHD equations in nearly optimal Sobolev spaces. \textit{Arch. Ration. Mech. Anal.} {\bf  223},  677--691  (2017)





\bibitem{HMRW85} \textsc{Holm, D.D., Marsden, J.E., Ratiu, T., Weinstein, A.}: Nonlinear stability of fluid and plasma equilibria. \emph{ Phys. Rep.} {\bf 123}(1–2), 116 (1985)


\bibitem{KK23}\textsc{Kim, H., Kwon, H.}: Global existence and uniqueness of weak solutions of a Stokes-Magneto system with fractional diffusions. \emph{J. Diff. Equat.} {\bf 374}, 497--547 (2023)



\bibitem{Li19}\textsc{Li, D.}: On Kato-Ponce and fractional Leibniz. \emph{ Rev. Mat. Iberoam.} {\bf 35}, 23--100 (2019)


 \bibitem{Li17}\textsc{Li J.L.,   Tan, W.K.,  Yin, Z.Y.}:Local existence and uniqueness for the non-resistive MHD equations in homogeneous Besov spaces.   \textit{Adv. Math.} {\bf 317},  786--798 (2017)

\bibitem{LL01}\textsc{Lieb, H E., Loss, M.}:   Analysis. volume {\bf 14}. American Mathematical Society (2001)

\bibitem{PLL} \textsc{Lions, P.-L.}: { Mathematical topics in fluid mechanics. Incompressible models}, Oxford Lecture Series in Mathematics and its Applications, {\bf 3} (1996)

\bibitem{Mo85} \textsc{Moffatt, H.K.}: Magnetostatic equilibria and analogous Euler flows of arbitrarily complex topology. I. Fundamentals. \emph{ J. Fluid Mech.} {\bf 159}, 359--378 (1985)

\bibitem{Mo21} \textsc{Moffatt, H.K.}: Some topological aspects of fluid dynamics. \emph{ J. Fluid Mech.} {\bf 914}:Paper No. P1 (2021)

\bibitem{Ni03} \textsc{Nishiyama, T.}: Magnetohydrodynamic approach to solvability of the three-dimensional stationary Euler equations. \emph{ Glasgow Math. J.} {\bf 44},  411--418 (2002)


\bibitem{NP18} \textsc{Novotn\'{y}, A., Petzeltov\'{a}, H.}: Weak Solutions for the Compressible Navier-Stokes Equations: Existence, Stability, and Longtime Behavior. Handbook of mathematical analysis in mechanics of viscous fluids, 1381--1546, Springer, Cham, (2018)

 



\bibitem{VCY89} \textsc{Vallis, G.K., Carnevale, G.F., Young, W.R.}: Extremal energy properties and construction of stable solutions of the Euler equations. \emph{ J. Fluid Mech.} {\bf 207}, 133--152 (1989)

\end{thebibliography}
\end{document}